\documentclass[reqno]{amsart}
\usepackage{amsaddr}
\usepackage[table]{xcolor}
\usepackage{pgf,pgfarrows,pgfnodes,pgfautomata,pgfheaps,pgfshade,hyperref, amssymb,enumerate,amsmath}
\usepackage{MnSymbol}
\usepackage[pagewise]{lineno}
\usepackage[all]{xy}
\usepackage[capitalize]{cleveref}
\usepackage{mathtools}
\usepackage[shortlabels]{enumitem}
\usepackage{stmaryrd}
\usepackage{bm}
\usepackage{boldline}

\newtheorem{theorem}{Theorem}[section]
\newtheorem*{theorem*}{theorem}
\newtheorem{proposition}[theorem]{Proposition}
\newtheorem{lemma}[theorem]{Lemma}
\newtheorem{corollary}[theorem]{Corollary}

\newtheorem{exem}[theorem]{Example}
\newtheorem{remark}[theorem]{Remark}

\usepackage{euscript,epsfig}
\usepackage{tikz}
\usetikzlibrary{graphs}
\usetikzlibrary[graphs]
\usetikzlibrary{arrows}
\usetikzlibrary{tikzmark}
\usetikzlibrary{positioning}
\usetikzlibrary{decorations.pathreplacing,calligraphy}
\usepackage{tikz-cd}
\usetikzlibrary{decorations.markings}
\usepackage[colorinlistoftodos]{todonotes}

\def\Per{{\rm Per}}
\author[1]{Jaime G\'omez}

\address[1]{Facultad de Matem\'aticas, Pontificia Universidad Cat\'olica de Chile\\ Santiago, Chile.}
\email[1]{jagomez7@uc.cl}

\thanks{J. Gómez was supported by ANID/ Doctorado Nacional No. 21200054}

\begin{document}

\title[Topo-isomorphisms of Irregular Toeplitz subshifts]{Topo-isomorphisms of Irregular Toeplitz subshifts for  residually finite groups.}

\subjclass[2020]{Primary: 37B05, 37B10} 

\keywords{Toeplitz subshift, topo-isomorphism, residually finite group, mean-equicontinuous, invariant measure}

\date{}

\begin{abstract}

For each countable residually finite group $G$, we present examples of irregular Toeplitz subshifts in $\{0,1\}^G$ that are topo-isomorphic extensions of its maximal equicontinuous factor. 
To achieve this, we first establish sufficient conditions for Toeplitz subshifts to have invariant probability measures as limit points of periodic invariant measures of $\{0,1\}^G$. Next, we demonstrate that the set of Toeplitz subshifts satisfying these conditions is non-empty. When the acting group $G$ is amenable, this construction provides non-regular extensions of totally disconnected metric compactifications of $G$ that are (Weyl) mean-e\-qui\-con\-ti\-\-nuous dynamical systems.
\end{abstract}

\maketitle

\section{Introduction}\label{sec:intro}

The mean-e\-qui\-con\-ti\-\-nuous systems have been defined in \cite{LiTuYe15} for actions on $\mathbb{Z}$, and studied in the case of amenable groups in \cite{FuGrLe22} and \cite{LaSt18}. 
In the context of amenable groups, mean-equicontinuity  measures how closely a system resembles equicontinuity under certain Weyl-pseudometrics  associated with F\o lner sequences.
In this case, mean-equicontinuity is equivalent to being a topo-isomorphism extension of the maximal e\-qui\-con\-ti\-\-nuous factor, as proven in \cite{DoGl16} for actions on $\mathbb{Z}$ in the minimal case, and later extended to actions of amenable groups in \cite{FuGrLe22}.

\medskip 

Toeplitz subshifts were originally defined in \cite{JaKe69} and have been extensively studied due to their properties related to entropy, representation of Choquet simplices, among others (See for instance \cite{Wi84}, \cite{Do05}, \cite{GJ00}).
In \cite{DoLa98}, Downarowicz and Lacroix characterized Toeplitz subshifts as minimal symbolic almost $1$-$1$ extensions of odometers, which in turn implies that the maximal e\-qui\-con\-ti\-\-nuous factor of these systems are odometers. Within this class of subshifts, there exist regular Toeplitz subshifts, which were studied in \cite{Wi84} and provide examples of topo-isomorphic extensions of their respective maximal e\-qui\-con\-ti\-\-nuous factor.

\medskip

In subsequent years, Toeplitz subshifts were defined for actions of more general groups and  characterized as minimal symbolic almost $1$-$1$ extensions of $G$-odometers (see \cite{Co06},\cite{CoPe08}.\cite{Kr10}). 
The notion of regular Toeplitz subshift was also extended (see \cite{LaSt18} for amenable groups and \cite{CeCoGo23} for general residually finite groups).

Regular Toeplitz subshifts are examples of topo-isomorphic extensions of their maximal e\-qui\-con\-ti\-\-nuous factors, as we mention in Proposition \ref{regular-topo}.

\medskip

Despite the existence of examples of Toeplitz subshifts that are topo-isomorphic to their maximal e\-qui\-con\-ti\-\-nuous factors, these examples have so far been limited to the regular ones. The main challenge in finding new examples lies in the search for uniquely ergodic Toeplitz subshifts that are not regular. 
Downarowicz and Lacroix constructed an irregular example of this kind for $\mathbb{Z}$-actions in \cite{DoKa15}.
When the acting group is not $\mathbb{Z}$ the challenge is the lack of a well-behaved decomposition of some Toeplitz subshifts, as noted in \cite{Wi84} and further developed in \cite{Do05}.  This naturally leads us to the question of whether  similar examples exist when the acting group is not necessarily $\mathbb{Z}$.

\medskip

In this document, we construct irregular Toeplitz subshifts that are topo-isomorphic extensions of their maximal e\-qui\-con\-ti\-\-nuous factors. This construction is carried out for residually finite groups, with no assumption of amenability over the group.
This work also complements the findings in \cite{FuGrLe22}, as it provides new examples of mean-e\-qui\-con\-ti\-\-nuous systems over amenable residually finite groups beyond  $\mathbb{Z}^d$. 
To construct these examples, we establish sufficient conditions for a Toeplitz subshift in $\Sigma^G$ to have an invariant probability measure, which is a limit point of a sequence of periodic invariant probability measures over $\Sigma^G$, as shown in Proposition \ref{Supp-mu}. 
Subsequently, we construct an irregular Toeplitz array satisfying the hypothesis of Proposition \ref{Supp-mu}, leading to a subshift associated, which is a  topo-isomorphic extension of a totally disconnedted metric compactification. This result is stated with more precision as Theorem \ref{theo:main1}, which is the main Theorem of this document.

\begin{theorem}\label{theo:main1}
    Let $G$ be a countable, residually finite group and let $\overleftarrow{G}$ be a totally disconnected metric compactification of $G$.
    Then, there exists an \textit{irregular} Toeplitz element $\eta\in\{0,1\}^G$ such that the maximal e\-qui\-con\-ti\-nuous factor is $\overleftarrow{G}$, and $\overline{O_\sigma(\eta)}$ is a topo-isomorphic extension of $\overleftarrow{G}$.
\end{theorem}

Corollary \ref{coro:main2} follows immediately from Theorem \ref{theo:main1} and \cite[Theorem 1.1]{FuGrLe22}.

\begin{corollary}\label{coro:main2}
    Let $G$ be an amenable countable residually finite group and let $\overleftarrow{G}$ be a totally disconnected metric compactification of $G$. Then, there exists an irregular mean-e\-qui\-con\-ti\-\-nuous Toeplitz subshift whose maximal e\-qui\-con\-ti\-\-nuous factor is $\overleftarrow{G}$.
\end{corollary}

The structure of this document is as follows: In Section \ref{sec:basicdef}, we provide basic facts related to topological dynamical systems and measure theory, along with a brief introduction to residually finite groups, $G$-odometers, and Toeplitz subshifts. 
In Section \ref{Sectio-3}, we introduce a new representation of the set of periods of some Toeplitz arrays, which is essential for the proof of Proposition \ref{Supp-mu} at the end of that Section. 
This proposition guarantees that the irregular Toeplitz array $\eta$ defined in Section \ref{sec:irregular-construction} generates a Toeplitz subshift having at least one invariant probability measure. 
In the last part of Section \ref{sec:irregular-construction}, we define a map from the set of invariant measures of the Toeplitz subshift defined in Section \ref{sec:irregular-construction} to a one-point set, as stated in Proposition \ref{at-least}. 
This map plays a crucial role in the proof of Theorem \ref{theo:main1}, which is presented in Section \ref{sec:Topo}, where we guarantee that there exists an isomorphisms between the Toeplitz subshift previously constructed and its maximal equicontinuous factor and finally, we prove that the subshift is uniquely ergodic.
\section{Preliminaries}\label{sec:basicdef}

\subsection{Topological dynamical systems}\label{subsec:defrfgroups}
In this article,  $G$ always refers to  a countable discrete infinite group, where $1_G$ denotes the neutral element of $G$. We define a {\it topological dynamical system} (or {\it dynamical system}) $(X,\sigma, G)$ as a system where $\sigma$ is a left continuous action of $G$ on a compact metric space $X$. The dynamical system $(X,\sigma, G)$ is {\it minimal} if for every $x\in X$, its orbit $O_\sigma(x)=\{\sigma^gx: g\in G\}$ is dense in $X$. The dynamical system is {\it e\-qui\-con\-ti\-\-nuous} if the collection of maps $\{\sigma^g\}_{g\in G}$ is e\-qui\-con\-ti\-\-nuous.
A topological dynamical system $(X,\sigma,G)$ is called {\it free} if $\sigma^gx=x$ implies $g=1_G$ for every $x\in X$.

An {\it invariant measure} of the topological dynamical system $(X,\sigma,G)$ is 
a probability measure $\mu$ defined on the Borel sigma-algebra of $X$, satisfying the condition that for every $g\in G$ and for every Borel set $A\subseteq X$, we have $\mu(A)=\mu(\sigma^g A)$.
A Borel set is called {\it invariant} if $\sigma^g A=A$ for every $g\in G$. 
An invariant measure $\mu$ is said to be \textit{ergodic} if $\mu(A)\in\{0,1\}$ for every invariant set $A\subseteq X$. The set of invariant measures of $(X,\sigma,G)$ is denoted $M_G(X)$, which is a metrizable Choquet simplex with extreme points the ergodic measures of $(X,\sigma, G)$ (see for instance \cite{CoPe14}).
When the system has a unique invariant measure, it is called {\it uniquely ergodic}. 
Using \cite[Theorem 3.6]{Au88} and \cite[Proposition 8.1]{BeOh07} one can prove that every equicontinuous minimal dynamical system is uniquely ergodic. If $\mu$ is an invariant measure of $(X,\sigma,G)$, then $(X,\sigma,G,\mu)$ is referred as a {\it probability measure-preserving} dynamical system.

Let $(X,\sigma, G)$ and $(Y,\phi,G)$ be two topological dynamical systems. A continuous surjective map $\pi: X\to Y$ is called a {\it factor map} if $\pi(\sigma^g x)=\phi^g\pi(x)$, for every $x\in X$ and every $g\in G$.  In this case, we say that $(X,\sigma, G)$ is an {\it extension} of $(Y,\phi,G)$ and $(Y,\phi, G)$ is a {\it factor} of $(X,\sigma,G)$. 
Given a factor map $\pi$ from $(X,\sigma,G)$ to $(Y,\phi,G)$ and $\mu$ being an invariant measure of $(X,\sigma,G)$, there exists an invariant measure of $(Y,\phi,G)$ associated to $\mu$  called the  {\it pushforward measure} of $\mu$ by $\pi$. It is defined as $\pi(\mu)(B)=\mu(\pi^{-1}(B))$ for every $B\subseteq Y$ Borel set. 
A factor map $\pi$ between $(X,\sigma, G)$ and $(Y,\phi, G)$ is said to be {\it almost one to one} (or {\it almost $1$-$1$}) if the set 
$\{y\in Y: |\pi^{-1}(\{y\})|=1\}\subseteq Y$ is residual.  We say that  $(X,\sigma,G)$ is an {\it almost 1-1 extension} of $(Y,\phi,G)$. 
Recall that when the system $(Y,\phi, G)$ is minimal, being {\it almost one to one} is equivalent to the existence of $y\in Y$ such that $|\pi^{-1}\{y\}|=1$ (see \cite{Ve70}).

An e\-qui\-con\-ti\-\-nuous system $(Y,\phi,G)$ is the {\it maximal e\-qui\-con\-ti\-\-nuous factor} of a system $(X,\sigma,G)$ if there is a factor map $\pi: X\to Y$ such that for every other factor map $f: X\to Y'$, where $(Y',\phi',G)$ is an e\-qui\-con\-ti\-\-nuous system, there exists a factor map $\overline{f}: Y\to Y'$ such that $f=\overline{f}\circ \pi$.
It is well known that if $(X,\sigma,G)$ is a minimal almost $1$-$1$ extension of a minimal e\-qui\-con\-ti\-\-nuous system $(Y,\phi,G)$, then $(Y,\phi,G)$ is the maximal e\-qui\-con\-ti\-\-nuous factor of $(X,\sigma,G)$ (see \cite{Kr10}).

Two probability measure-preserving dynamical systems $(X,\sigma,G,\mu)$ and $(Y,\phi,G,\nu)$ are {\it measure conjugate }if there exists a factor map $h: X\to Y$, and invariant conull sets $X'\subseteq X$ and $Y'\subseteq Y$ such that  $h|_{X'}: X'\to Y'$ is a  bimeasurable bijection, and $\nu=h|_{X'}(\mu)$. We say that $h|_{X'}$ is a {\it measure conjugacy $(\bmod\; 0)$} and $h$ is called a {\it measure conjugacy of $\mu$}. A dynamical system $(X,\sigma,G)$ is a {\it topo-isomorphic extension} of $(Y,\phi,G)$ if there is a factor map $h: X\to Y$ such that $h$ is a measure conjugacy of $\mu$ for every $\mu$ invariant measure of $(X,\sigma, G)$. In this case, $h$ is called a {\it topo-isomorphism}.

A group $G$ is said to be {\it amenable} if there exists a sequence $(F_n)_{n\in\mathbb{N}}$ of non-empty finite subsets of $G$, called a {\it (left) F\o lner sequence}, verifying
\begin{align*}
    \lim_{n\to\infty} \dfrac{|gF_n\triangle F_n|}{|F_n|}=0\; \mbox{ for all }\; g\in G,
\end{align*}
where $|\cdot|$ denotes the cardinal of a set and $\triangle$ represents the symmetric difference. 
This definition of amenable group is equivalent to the existence of invariant probability measures for every topological dynamical system $(X,\phi,G)$ (See \cite{CC10} for a more detailed presentation).
If $G$ is an amenable group, a dynamical system $(X,\sigma,G)$ is {\it mean-e\-qui\-con\-ti\-\-nuous} if for all $\varepsilon>0$ there exists $\delta_\varepsilon>0$ such that for every $x,z\in X$ with $d(x,z)<\delta_\varepsilon$ we have
\begin{align*}
    \sup\{\limsup_{n\to\infty}\dfrac{1}{|F_n|}\sum_{g\in F_n} d(\sigma^g x,\sigma^g z):\mathcal{F}=(F_n)_{n\in\mathbb{N}}\mbox{ is a F\o lner sequence}\}<\varepsilon,
\end{align*}
where $d$ denotes a metric of $X$. 
The following Theorem states the relation between mean-equicontinuity and topo-isomorphism. 
This Theorem holds in a broader context where $G$ is not necessarily a countable discrete group.
\begin{theorem}[{\cite[Theorem 1.1]{FuGrLe22}}]\label{FuGrLe}
    Let $G$ be a locally compact sigma-compact amenable group. A topological dynamical system $(X,\sigma, G)$ is mean-e\-qui\-con\-ti\-\-nuous if and only if it is a topo-isomorphic extension of its maximal e\-qui\-con\-ti\-\-nuous factor.
\end{theorem}

\subsection{$G$-odometers and residually finite groups}

In this subsection, we provide definitions and basic properties of $G$-odometers and residually finite groups. For a more comprehensive presentation of these topics, we refer to \cite{CoPe08} and \cite{Kr10}.

We say that $G$ is {\it residually finite}, if there exists a decreasing sequence $(\Gamma_n)_{n\in\mathbb{N}}$ of finite index subgroups of $G$ with trivial intersection. We can assume that $\Gamma_n$ is normal for every $n\in\mathbb{N}$ (see \cite{CC10}).
Let $G$ be an infinite residually finite group and let $(\Gamma_n)_{n\in\mathbb{N
}}$ be a decreasing sequence of finite index subgroups of $G$ with trivial intersection. 
We define the {\it $G$-odometer associated to $(\Gamma_n)_{n\in\mathbb{N}}$}, denoted as $\overleftarrow{G}$, as the inverse limit of the inverse system given by $(G/\Gamma_n,\varphi_n)$, where $\varphi_n\colon G/\Gamma_{n+1}\to G/\Gamma_n$ is the canonical map given by $\Gamma_{n+1}\subseteq \Gamma_n$, i.e.,
\begin{align*}
    \overleftarrow{G}:=\varprojlim(G/\Gamma_n,\varphi_n)=\{(x_n)_{n\in\mathbb{N}}\in\prod_{n\in\mathbb{N}} G/\Gamma_n\mid x_n=\varphi_n(x_{n+1}) \mbox{ for every }n\in \mathbb{N}\}.
\end{align*}
 $\overleftarrow{G}$ has the induced topology of $\prod_{n\in\mathbb{N}}G/\Gamma_n$, where each $G/\Gamma_n$ is endowed with the discrete topology.
In this scenario, $\overleftarrow{G}$ is a Cantor set.
There is also a canonical continuous action $\phi$ of $G$ on $\overleftarrow{G}$ given by the left coordinate-wise multiplication. 
The dynamical system $(\overleftarrow{G},\phi,G)$ is an e\-qui\-con\-ti\-\-nuous, minimal Cantor system, and hence, uniquely ergodic with measure $\nu$. When every $\Gamma_n$ is normal, $\overleftarrow{G}$ is a subgroup of $\prod_{n\in\mathbb{N}}G/\Gamma_n$, and $G$ is identified with a dense subgroup of $\overleftarrow{G}$. In this case, the left Haar  measure $\nu$ is the unique invariant measure of the system.

A {\it totally disconnected metric compactification of $G$} is a totally disconnected compact group 
for which an injective homomorphism  $\iota: G\to\overleftarrow{G}$ exists and its image is dense in $\overleftarrow{G}$. 
Note that every $G$-odometer having a group structure is a totally disconnected metric compactification of $G$, and conversely, every totally disconnected metric compactification of $G$ is a $G$-odometer (see \cite[Lemma 2.1]{CeCoGo23}).
%It is known that $G$-odometers associated to subsequences of $(\Gamma_n)_{n\in\mathbb{N}}$ are conjugate as dynamical systems (see \cite[Lemma 2]{CoPe08}).

\subsection{Toeplitz subshifts}\label{subsec:deftoeplitzodometers}
The definitions and results written in this subsection can be found in \cite{CoPe08}. 
 Let  $\Sigma$ be a finite set with $|\Sigma|\geq 2$. The set $$\Sigma^G=\{x=(x(g))_{g\in G}: x(g)\in \Sigma, \mbox{ for every } g\in G\}$$  is a Cantor set if  $\Sigma$ has the discrete topology  and $\Sigma^G$ is endowed with the product topology. 
 We can associate with $\Sigma^G$ the continuous action known as {\it (left) shift action} $\sigma$ of $G$ given by
 \begin{align*}
     \sigma^g x(h)=x(g^{-1}h)\mbox{ for every }h,g\in G\mbox{ and }x\in \Sigma^G.
 \end{align*}
 
 The topological dynamical system $(\Sigma^G, \sigma, G)$ is known as a {\it full G-shift}. 
 A subset $X\subseteq \Sigma^G$ is a {\it subshift} of $\Sigma^G$ if it is closed and invariant.
 The topological dynamical system $(X, \sigma|_X,G)$ is also called a {\it subshift}.
 Let $x\in \Sigma^G$ and $\Gamma\subseteq G$ be a subgroup of finite index. We define  
\begin{align*}
    \Per(x,\Gamma,\alpha)&=\{g\in G: x(\gamma g)=\alpha \mbox{ for every }\gamma\in\Gamma\}, \mbox{ for every } \alpha\in \Sigma.\\
    \Per(x,\Gamma)&=\bigcup_{\alpha\in \Sigma}\Per(x,\Gamma,\alpha).
\end{align*}

An element $\eta\in \Sigma^G$ is called a \textit{Toeplitz array} if for every $g\in G$ there exists a finite index subgroup $\Gamma$ of $G$ such that $g\in \Per(\eta,\Gamma)$. The finite index subgroup $\Gamma$ is a {\it group of periods} of $\eta$ if $\Per(\eta,\Gamma)\neq \emptyset$. A group of periods $\Gamma$ of $\eta$ is an {\it essential group of periods} of $\eta$ if $\Per(\eta,\Gamma,\alpha)\subseteq g\Per(\eta,g^{-1}\Gamma g,\alpha)=\Per(\sigma^g\eta,\Gamma,\alpha)$ for every $\alpha\in \Sigma$ implies $g\in \Gamma$.
 
  Let $\eta\in \Sigma^G$ be a Toeplitz array. It is possible to construct  a {\it period structure}  $(\Gamma_n)_{n\in\mathbb{N}}$ of $\eta$, i.e., a decreasing sequence $(\Gamma_n)_{n\in\mathbb{N}}$ of finite index subgroups of $G$ such that $\Gamma_n$ is an essential group of periods of $\eta$ for every $n\in\mathbb{N}$ and $G=\bigcup_{n\in\mathbb{N}}\Gamma_n$ (see \cite[Corollary 6]{CoPe08}).

A subshift $X$ is called  a {\it Toeplitz subshift}, if there exists a Toeplitz array $\eta\in\Sigma^G$  such that $X=\overline{O_\sigma(\eta)}$. Recall that these dynamical systems are minimal (see \cite{CoPe08}).
Let $\eta\in \Sigma^G$ be a Toeplitz array, and let $(\Gamma_n)_{n\in\mathbb{N}}$ be a period structure of $\eta$. Let $\overline{O_{\sigma}(\eta)}$ be its associated Toeplitz subshift.
For each $n\in\mathbb{N}$, we define
 $$C_n=\{ x\in \overline{O_\sigma(\eta)}:  \Per(x,\Gamma_n, \alpha)=\Per(\eta, \Gamma_n, \alpha), \mbox{ for every } \alpha\in \Sigma\}.$$
Lemma 2.5 in \cite{CeCoGo23} provides that for every $n\in \mathbb{N}$ and for all $g,h\in G$, we have that $\sigma^g C_n=\sigma^h C_n$ if and only if $g\Gamma_n=h\Gamma_n$, and consequently, $\{\sigma^{g^{-1}} C_n:g\in D_n\}$ is a clopen partition of $\overline{O_\sigma(\eta)}$.
The stabilizer of $x\in \overline{O_\sigma(\eta)}$ is given by $\bigcap_{n\in\mathbb{N}} v_n\Gamma_n v_n^{-1}$, where $v_n\in G$ is such that $x\in \sigma^{v_n}C_n$. When  $\Gamma_n$ is normal for every $n\in\mathbb{N}$,  the Toeplitz subshift $(\overline{O_\sigma(\eta)},\sigma, G)$ is free if and only if $\bigcap_{n\in\mathbb{N}}\Gamma_n=\{1_G\}$.  

The relationship between Toeplitz subshifts and countable residually finite groups is given by the following statement:  there exists a free Toeplitz subshift or a free $G$-odometer if and only if $G$ is residually finite (see for example \cite{CoPe08} and \cite{Kr10}). 

From now on, let $G$ denote a countable discrete infinite residually finite group, and let $(\Gamma_n)_{n\in\mathbb{N}}$ denote a decreasing sequence of normal subgroups of finite index of $G$ with trivial intersection.

The relation of Toeplitz subshifts and $G$-odometers is represented in the following Proposition. 

\begin{proposition}[\cite{CoPe08}]\label{pimap}
    Let $\eta\in\Sigma^G$ be a Toeplitz array and let $\overline{O_\sigma(\eta)}$ be the Toeplitz subshift associated to $\eta$. Suppose that $(\Gamma_n)_{n\in\mathbb{N}}$ is a period structure of $\eta$. 
    Let $\overleftarrow{G}$ be the $G$-odometer associated to $(\Gamma_n)_{n\in\mathbb{N}}$. 
    Then, the map $\pi:\overline{O_\sigma(\eta)}\to \overleftarrow{G}$ defined by
        $\pi(x)=(g_n\Gamma_n)_{n\in\mathbb{N}}$, where $x\in\sigma^{g_n}C_n$ for every $n\in\mathbb{N}$,
    is an almost $1$-$1$ factor map. Consequently, $\overleftarrow{G}$ is the maximal e\-qui\-con\-ti\-\-nuous factor of $\overline{O_\sigma(\eta)}$. Moreover, 
    \begin{align*}
        \mathcal{T}=\{x\in\overline{O_\sigma(\eta)}: x\mbox{ is a Toeplitz array}\}=\pi^{-1}\{y\in\overleftarrow{G}: |\pi^{-1}\{y\}|=1\}.
    \end{align*}
\end{proposition}

The following Proposition is considered as the  non-amenable version of \cite[Lemma 5]{CoPe14} and it plays an important role in the development of this paper.
\begin{proposition}[{\cite[Lemma 2.8]{CeCoGo23}}]\label{decom}
There exist an increasing sequence $(t_i)_{i\in \mathbb{N}}\subseteq \mathbb{N}$ and a sequence $(D_i)_{i\in \mathbb{N}}$ of finite subsets of $G$ such that for every $i\in \mathbb{N}$,
\begin{enumerate}
    \item $\{1_G\}\subseteq D_i\subseteq D_{i+1}$.
    \item $D_i$ is a fundamental domain of $G/\Gamma_{t_i}$, i.e., $D_i$ contains a unique element of every coset in $G/\Gamma_{t_i}$.
    \item $G=\bigcup_{i=1}^\infty D_i$.
    \item $D_j=\bigcup_{v\in D_j\cap \Gamma_{t_i}} vD_i$, for each $j>i\geq 1$.
\end{enumerate}
\end{proposition}

\subsection{Regular Toeplitz subshifts.}\label{Regular-T} 
See \cite[Section 3]{CeCoGo23} for a further development of this subsection. 
Let $\Sigma$ be a finite set with $|\Sigma|\geq 2$. Let $\eta\in\Sigma^G$ be a Toeplitz array such that $(\Gamma_n)_{n\in\mathbb{N}}$ is a period structure of $\eta$.
Recall that  the  Toeplitz subshift $\overline{O_{\sigma}(\eta)}$ is an almost 1-1 extension of the  $G$-odometer $\overleftarrow{G}$ associated with $(\Gamma_n)_{n\in\mathbb{N}}$, and $\mathcal{T}$ denotes the set of Toeplitz arrays in $\overline{O_\sigma(\eta)}$. 
Let $(D_n)_{n\in\mathbb{N}}$ be a sequence as in Proposition \ref{decom}, and if necessary, by taking a subsequence of $(\Gamma_i)_{n\in\mathbb{N}}$, we can suppose that $t_i=i$. 
The sequence $(d_i)_{i\in\mathbb{N}}$, given by $d_i=\dfrac{|D_i\cap \Per(\eta,\Gamma_i)|}{|D_i|}$, is increasing and convergent. We denote $d=\lim_{i\to\infty}d_i$. The Toeplitz array $\eta$ is said to be {\it regular} if any of the following equivalent statements is true: 
\begin{enumerate}
\item $\nu(\pi(\mathcal{T}))=1$.
\item There exists an invariant probability measure $\mu$ of $(\overline{O_\sigma(\eta)}, \sigma, G)$ such that $\mu(\mathcal{T})=1$.
\item There exists a unique invariant probability measure $\mu$ of $(\overline{O_\sigma(\eta)},\sigma, G)$ such that  $\mu(\mathcal{T})=1$.
\item $d=1$.
 \end{enumerate}
If $\eta$ does not satisfy any of the previous statements, we call it {\it irregular} or {\it non-regular}.

As an immediate consequence of the above discussion, we get the following Proposition.
\begin{proposition}\label{regular-topo}
    Every free regular Toeplitz subshift is a topo-isomorphic extension of its maximal e\-qui\-con\-ti\-\-nuous factor.
\end{proposition}

\section{Toeplitz subshifts with invariant probability measures}\label{Sectio-3}

 Let $\Sigma$ be a finite set such that $|\Sigma|\geq 2$ and let $(\Gamma_n)_{n\in\mathbb{N}}$ be a decreasing sequence of normal finite index subgroups of $G$ with $\bigcap_{i\in\mathbb{N}}\Gamma_i=\{1_G\}$.
Let $(D_n)_{n\in\mathbb{N}}$ be an increasing sequence as in Proposition \ref{decom} with $(t_n)_{n\in\mathbb{N}}$ as (possibly after taking a subsequence)  $t_n=n$ for every $n\in\mathbb{N}$. For every $n\in\mathbb{N}$, define 
\begin{align}\label{J(n)}
    J(n)=D_n\setminus\bigcup_{i=0}^{n-1} J(i)\Gamma_{i+1},
\end{align}
where $J(0)=\{1_G\}$. Observe that $D_n\subseteq \bigcup_{i=0}^{n}J(i)\Gamma_{i+1}$.

\medskip

\begin{proposition}\label{auxiliar0}
    For every $n\in\mathbb{N}$, we have 
    \begin{align*}
        J(n+1)=\bigcup_{\gamma\in(D_{n+1}\cap \Gamma_n)\setminus \{1_G\}}\gamma J(n).
    \end{align*}
\end{proposition}
\begin{proof}
    If $u\in D_{n+1}$, then $u=\gamma v$ for some $\gamma\in D_{n+1}\cap \Gamma_n$ and $v\in D_n$. If $u\in J(n+1)$, using (\ref{J(n)}), we obtain that $\gamma\neq 1_G$. Additionally, $v\notin \bigcup_{i=0}^{n-1} J(i)\Gamma_{i+1}$. Otherwise, by (\ref{J(n)}) and normality of $\gamma_n$, we would have that $\gamma v\in \bigcup_{i=0}^{n} J(i)\Gamma_{i+1}$, which is not possible. Now, suppose that $v\in J(n)$ and $\gamma\in (D_{n+1}\cap\Gamma_n)\setminus\{1_G\}$. We will show that $\gamma v\in J(n+1)$. If $\gamma v\in \bigcup_{i=0}^n J(i)\Gamma_{i+1}$, we obtain that $\gamma v\in J(n)\Gamma_{n+1}$.
    Therefore, we can conclude that $\gamma\in\Gamma_{n+1}\cap D_{n+1}=\{1_G\}$, a contradiction.
    Thus, the proposition is proven. 
\end{proof}

\begin{proposition}\label{Per-eq}
Let $\eta\in\Sigma^G$ be a Toeplitz sequence with period structure  $(\Gamma_n)_{n\in\mathbb{N}}$. Then, $J(n)\subseteq \Per(\eta,\Gamma_{n+1})\setminus\Per(\eta,\Gamma_n)$ for every $n\in\mathbb{N},$ if and only if
    \begin{align}\label{T1}
        \Per(\eta,\Gamma_n)=\bigcup_{i=0}^{n-1} J(i)\Gamma_{i+1} \mbox{ for every }n\in\mathbb{N}.
    \end{align}
\end{proposition}
\begin{proof}
Assume $J(n)\subseteq \Per(\eta,\Gamma_{n+1})\setminus \Per(\eta,\Gamma_n)$ for every $n\in\mathbb{N}$.
We use induction on $n$: $n=0$. Note that $\Gamma_1\subseteq \Per(\eta,\Gamma_{1})$. If there exists $g\in \Per(\eta,\Gamma_1)\setminus \Gamma_1$, we can find $d\in D_1$ and $\gamma\in\Gamma_1$ such that $g=d\gamma$. Since $g\in\Per(\eta,\Gamma_1)$, we obtain that $d\in\Per(\eta,\Gamma_1)$. On the other hand, $d\in D_1\setminus\Gamma_1=J(1)\subseteq \Per(\eta,\Gamma_2)\setminus \Per(\eta,\Gamma_1) $, a contradiction. Then, $\Per(\eta,\Gamma_1)=\Gamma_1$. 

Suppose that (\ref{T1}) is true for $n=k$, we will prove it for $n=k+1$. By hypothesis of induction, we have that 
\begin{align*}
    \bigcup_{i=0}^{k-1} J(i)\Gamma_{i+1}=\Per(\eta,\Gamma_k)\subseteq\Per(\eta,\Gamma_{k+1}).
\end{align*}

Moreover, $J(k)\Gamma_{k+1}\subseteq\Per(\eta,\Gamma_{k+1})$. 
Therefore, $\bigcup_{i=0}^{k}J(i)\Gamma_{i+1}\subseteq \Per(\eta,\Gamma_{k+1}).$
Suppose that there exists
$g\in\Per(\eta,\Gamma_{k+1})\setminus\bigcup_{i=0}^{k}J(i)\Gamma_{i+1}$. We know that $g=d\gamma$, for some $d\in D_{k+1}$ and $\gamma\in \Gamma_{k+1}$ 
Thus, $d\in \Per(\eta,\Gamma_{k+1})\cap J(k+1)$, which is not possible.

Now, assume that (\ref{T1}) is true. For every $n\in\mathbb{N}$ observe that
\begin{align*}
    \Per(\eta,\Gamma_{n+1})\setminus\Per(\eta,\Gamma_n)=\left(\bigcup_{i=0}^n J(i)\Gamma_{i+1}\right)\setminus\left( \bigcup_{i=0}^{n-1} J(i)\Gamma_{i+1}\right)=J(n)\Gamma_{n+1}.
\end{align*}
Therefore, $J(n)\subseteq \Per(\eta,\Gamma_{n+1})\setminus\Per(\eta,\Gamma_n)$.
\end{proof}

\begin{lemma}[{\cite[Lemma 4.7]{CeCoGo23}}]\label{good-relation1}
Let $n\in\mathbb{N}$. For every $m\geq n+2$ there exists 
\begin{equation}\label{good-relation}
\gamma\in (\Gamma_{n+1}\cap D_m)\setminus (D_{n+1}\Gamma_{n+2} \cup  \cdots \cup D_{m-1}\Gamma_m).
\end{equation}
Moreover,
$$
\left|(\Gamma_{n+1}\cap D_m)\setminus (D_{n+1}\Gamma_{n+2} \cup  \cdots \cup D_{m-1}\Gamma_m)\right|\geq  \frac{|D_{m}|}{|D_{n+1}|}\prod_{l=1}^{m-n-1}\left(1-\frac{|D_{n+l}|}{|D_{n+l+1}|}   \right).
$$

Furthermore, if $\gamma$ satisfies (\ref{good-relation}), then
$$
\gamma D_{n+1}\subseteq  D_m\setminus (D_{n+1}\Gamma_{n+2} \cup  \cdots \cup D_{m-1}\Gamma_m).
$$
\end{lemma}

\medskip

 For every $n\in\mathbb{N}$, we define $\eta_n\in\Sigma^G$ as
$
    \eta_n(\gamma d)=\eta(d) \mbox{ for every }\gamma\in\Gamma_n, d\in D_n.$
    Note that $\sigma^{\gamma}\eta_n=\eta_n$ for every $\gamma\in\Gamma_n$. Therefore,
$   O_\sigma(\eta_n)=\{\sigma^{d^{-1}}\eta_n\mid d\in D_n\}$. 
We define a $\sigma$-invariant measure over $\Sigma^G$, associated with $\eta_n$, as follows 
\begin{align*}
    \mu_n=\dfrac{1}{|D_n|}\sum_{d\in D_n}\delta_{\sigma^{d^{-1}}\eta_n},
\end{align*}
where $\delta_x$ denotes the Dirac measure supported in $x$.
 We define the set $U_n$ as 
 \begin{equation}\label{definition}
 U_n:=\{x\in \Sigma^G: x(D_{n+1})=\eta_n(D_{n+1})\}.
 \end{equation}

The following Lemma can be regarded as a generalization of \cite[Lemma 4.9]{CeCoGo23},  since the Toeplitz array constructed in \cite[Section 4]{CeCoGo23} satisfies  ${J(n)\subseteq \Per(\eta,\Gamma_{n+1})\setminus \Per(\eta,\Gamma_n)}$ for every $n\in\mathbb{N}$.

\begin{lemma}\label{good-patches}
Let $\eta\in\Sigma^G$ be a Toeplitz sequence with a period structure $(\Gamma_n)_{n\in\mathbb{N}}$ such that $J(n)\subseteq \Per(\eta,\Gamma_{n+1})\setminus\Per(\eta,\Gamma_n)$ for every $n\in\mathbb{N}$.
Let $m,n\in \mathbb{N}$ with $m>n+2$.
If $\gamma_0 \in \Gamma_{m+1}\cap D_{n}$ satisfies the relation (\ref{good-relation}) and 
\begin{align}\label{Patching}
    \eta(\gamma_0\gamma u)=\eta(u), \mbox{ for every }u\in J(n)\mbox{ and }\gamma\in D_{n+1}\cap\Gamma_n,
\end{align}then $\sigma^{\gamma_0^{-1}}\eta\in U_{n}$. 
This implies that  $U_{n}\cap O_{\sigma}(\eta)\neq \emptyset.$  
\end{lemma}
\begin{proof}
 Let $\gamma_0\in \Gamma_{n+1}\cap D_m$ be an element satisfying the relation (\ref{good-relation}). 
    We aim to prove that $\sigma^{\gamma_0^{-1}}\eta\in U_n$.
     Lemma \ref{decom} implies $D_{n+1}=\bigcup_{\gamma\in D_{n+1}\cap\Gamma_n}\gamma D_n$, and (\ref{T1}) implies $D_n=(D_n\cap\Per(\eta,\Gamma_n))\cup J(n)$.
    Since $D_n\subseteq \Per(\eta,\Gamma_{n+1}) $, we have 
 \begin{equation}\label{eq2}
 \eta(\gamma_0u)=\eta(u) \mbox{ for }u\in D_n.
 \end{equation}
 For $u\in D_n\cap\Per(\eta,\Gamma_n)$ and $\gamma\in D_{n+1}\cap \Gamma_n$, we have that $\gamma_0\gamma\in \Gamma_n$
    and, as a result,
 \begin{equation}\label{eq3}
 \eta(\gamma_0\gamma u)=\eta(u).
 \end{equation}
  Let $u\in J(n)$ and $\gamma\in D_{n+1}\cap\Gamma_n\setminus\{1_G\}$. Thus, according to Lemma \ref{auxiliar0}, we have
 $\gamma_0\gamma u\in J(m)$. Since $\gamma_0$ satisfies (\ref{Patching}), we have
 \begin{align}\label{eq4}
     \eta(\gamma_0\gamma u)=\eta(u).
 \end{align}
 By combining (\ref{eq2}), (\ref{eq3}) and (\ref{eq4}), we deduce
 $$
 \eta(\gamma_0\gamma u)=\eta(u), \mbox{ for every }u\in D_n\mbox{ and } \gamma\in D_{n+1}\cap \Gamma_n. 
 $$ 
This implies that $\sigma^{\gamma_0^{-1}}\eta (w)=\eta(u)=\eta_n(w)$ for $w\in D_{n+1}$, $u\in D_n$ and $\gamma\in D_{n+1}\cap\Gamma_n$ such that $w=\gamma u$. This implies the result.
\end{proof}

 \begin{proposition}\label{Supp-mu}

 Let $\eta\in \Sigma^G$ be a Toeplitz sequence with a period structure $(\Gamma_n)_{n\in\mathbb{N}}$ such that $J(n)\subseteq \Per(\eta,\Gamma_{n+1})\setminus\Per(\eta,\Gamma_n)$ for every $n\in\mathbb{N}$.
If there is an increasing sequence  
 $\mathcal{M}=(n_k)_{k\in\mathbb{N}}\subseteq\mathbb{Z}^+$ such that for every $n_{s},n_{j}\in\mathcal{M}$,  with $n_{j}>n_{s}+2$, there exists $\gamma_0\in\Gamma_{n_{s}+1}\cap D_{n_{j}}$ that satisfies (\ref{good-relation}) and 
\begin{align*}
    \eta(\gamma_0\gamma u)=\eta(u) \mbox{ for every }\gamma\in\Gamma_{n_{s}}\cap D_{n_{s}+1} \mbox{ and  } u\in J(n_{s}),
\end{align*}
then, every limit point of $(\mu_{n_k})_{k\in\mathbb{N}}$ is supported in $\overline{O_\sigma(\eta)}$.
 \end{proposition}
 \begin{proof}
Let $\mu$ be a limit point of $(\mu_{n_k})_{k\in\mathbb{N}}$. 
Therefore, there exists a subsequence $\mathcal{N}=(n_{k_j})_{j\in\mathbb{N}}$ of $\mathcal{M}$ such that $(\mu_{n_{k_j}})_{j\in\mathbb{N}}$ converges to $\mu$, when $j\to\infty$.
The set of  cylinders of the form ${V=\{y\in\{0,1\}^G: y(s)=Q(s), s\in S\}}$, where $S$ is a finite subset of $G$ and $Q\in\{0,1\}^S$, forms a basis for the topology of $\{0,1\}^G$. If there exists a cylinder $V$ as described above with $\mu(V)>0$, then there exists $n\in\mathbb{N}$ and a cylinder 
\begin{align}\label{C}
    C=\{y\in\{0,1\}: y(v)=P(v), v\in D_n\}, P\in\{0,1\}^{D_n}
\end{align} 
satisfying $\mu(C)>0$. Let $C$ be a cylinder as in (\ref{C}) with $\mu(C)>0$. Our goal is to prove that there exists an element in $\overline{O_\sigma(\eta)}$ that belongs to $C$.
    Since the sequence $(\mu_{n_{k_j}})_{j\in\mathbb{N}}$ converges to $\mu$, when $j\to\infty$, there exists $j_0\in\mathbb{N}$ such that $\mu_{n_{k_l}}(C)>0$ for every $l\geq j_0$. Hence, ${O_\sigma(\eta_{n_{k_l}})\cap C\neq\emptyset}$. 
     
     Choose $l\geq j_0$ such that $n_{k_l}\geq n$.
     There exists $u\in D_{n_{k_l}}$ such that ${\sigma^{u^{-1}}\eta_{n_{k_l}}(v)=P(v)}$ for every $v\in D_n$. 
     Moreover,  observe that $uD_n\subseteq D_{n_{k_l}}\cdot D_n $, and we can assume that ${D_{n_{k_l}}\cdot D_n\subseteq D_{n_{k_l}+1}}$.  
    According to  Lemma \ref{good-patches}, there exists $g\in G$ such that $\sigma^{g^{-1}}\eta\in U_{n_{k_l}}$. 
     Therefore, $\sigma^{u^{-1}g^{-1}}\eta(v)=\sigma^{g^{-1}}\eta(uv)=\eta_{n_{k_l}}(uv)=P(v)$ for every $v\in D_n$, and we conclude.
 \end{proof}

\section{Irregular Toeplitz subshifts}\label{sec:irregular-construction}
Inspired by the ideas presented in \cite[Example 5.1]{DoKa15} and \cite{CeCoGo23}, we
provide a proof of Theorem \ref{theo:main1} in the remaining Sections.

\medskip

Let $(\Gamma_n)_{n\in\mathbb{N}}$ be a decreasing sequence of finite index normal subgroups  of $G$ with trivial intersection.
Let $(D_n)_{n\in\mathbb{N}}$ be a sequence of finite subsets of $G$ as in Lemma \ref{decom} with (possibly after taking a subsequence)  $t_i=i$. 

\medskip

The following Proposition characterizes the regularity of certain Toeplitz arrays.

\begin{proposition}\label{regular-eta} Let $\eta\in\Sigma^G$ be a Toeplitz sequence with period structure $(\Gamma_n)_{n\in\mathbb{N}}$ such that $J(n)\subseteq \Per(\eta,\Gamma_{n+1})\setminus\Per(\eta,\Gamma_n)$ for every $n\in\mathbb{N}$.
    For each $n\in\mathbb{N}$, we have that
    \begin{align*}
        1-d_{n+1}=\left(1-\dfrac{1}{|D_1|}\right)\prod_{j=1}^n\left(1-\dfrac{|D_j|}{|D_{j+1}|}\right),
    \end{align*}
    with $d_n$ defined as in subsection \ref{Regular-T}.
\end{proposition}
\begin{proof}
     Proposition \ref{Per-eq} implies 
     \begin{align*}
        [\Per(\eta,\Gamma_{n+1})\setminus \Per(\eta,\Gamma_n)]\cap D_{n+1}=J(n)=D_n\setminus(D_n\cap\Per(\eta,\Gamma_n)).
    \end{align*}
    Therefore,
    \begin{align*}
        d_{n+1}&=\dfrac{|D_{n+1}\cap\Per(\eta,\Gamma_{n+1})|}{|D_{n+1}|}\\
        &=\dfrac
{|D_n\cap\Per(\eta,\Gamma_n)||D_{n+1}\cap \Gamma_n|}{|D_{n+1}|}+\dfrac{|D_n|-|D_n\cap\Per(\eta,\Gamma_n)|}{|D_{n+1}|}\\
&=\dfrac{|D_n\cap\Per(\eta,\Gamma_n)|}{|D_n|}+\dfrac{|D_n|}{|D_{n+1}|}\left(1-\dfrac{|D_n\cap\Per(\eta,\Gamma_n)|}{|D_n|}\right)\\
&=d_n+\dfrac{|D_n|}{|D_{n+1}|}(1-d_n).
    \end{align*}
    Thus,
    \begin{align*}
        1-d_{n+1}=(1-d_n)\left(1-\dfrac{|D_n|}{|D_{n+1}|}\right).
    \end{align*}
    Using  induction on $n$ we can conclude.
\end{proof}
\begin{remark}\label{irregularity}
Recall that 
\begin{align*}
    0<\left(1-\dfrac{1}{|D_1|}\right)\prod_{j=1}^\infty\left(1-\dfrac{|D_j|}{|D_{j+1}|}\right)\leq 1\mbox{ if and only if }\dfrac{1}{|D_1|}+\sum_{j=1}^\infty \dfrac{|D_{j}|}{|D_{j+1}|}\mbox{ converges}.
\end{align*}
Moreover, a computation provides that if $L:=\dfrac{1}{|D_1|}+\sum_{j=1}^\infty \dfrac{|D_j|}{|D_{j+1}|}<\infty$, then
\begin{align*}
    e^{-2L}\leq \left(1-\dfrac{1}{|D_1|}\right)\prod_{j=1}^\infty \left(1-\dfrac{|D_j|}{|D_{j+1}}\right)\leq 1.
\end{align*}

Therefore, Proposition \ref{regular-eta} and Subsection \ref{Regular-T} imply that $\eta$ is regular if and only if the series $L$ diverges.
\end{remark}

\subsection{Construction of irregular Toeplitz arrays}\label{Irregular-arr}

From now on, we  assume that the decreasing sequence of normal finite index subgroups $(\Gamma_n)_{n\in\mathbb{N}}$ of $G$ satisfies: 
\begin{enumerate}
    \item  $L=\frac{1}{|D_1|}+\sum_{j=1}^\infty \frac{|D_j|}{|D_{j+1}|}$  converges.
\item\label{LL}
    $1-e^{-2L}<\frac{1}{4}$.
\end{enumerate}

\medskip

 We define $\eta\in\{0,1\}^G$ as follows:

\medskip

{\bf 1\textsuperscript{st} Step:} Set $J(0)=\{1_G\}$ and let $\eta(g)=1$ for every $g\in\Gamma_1$.\\

{\bf 2\textsuperscript{nd} Step:} We define $J(1)=D_1\setminus \Gamma_1$ and $\eta(g)=0$ for every $g\in J(1)\Gamma_2$. Let $m(1)=|J(1)|$. Consider $J(1)=\{g_1^1,g_2^1,\ldots,g_{m(1)}^1\}$.  \\

{\bf\bm{$s+1$}\textsuperscript{th} Step:} There exist $k\in\mathbb{N}$ and $0\leq s'\leq m(k)$ such that $s=m_{k-1}+s'$, which implies that $m_{k-1}\leq s<m_k$, where \begin{align}\label{n-k}
    m_k:=1+k+\sum_{i=0}^km(i),
\end{align} with $m(i):=|J(i)|$ and $m(0)=1$ for every $0\leq i\leq k$.  
Let $J(s)=D_s\setminus (\bigcup_{i=0}^{s-1}J(i)\Gamma_{i+1})$, and $J(k)=\{g_1^{k},g_2^{k},\ldots, g_{m(k)}^{k}\}$. 

 If $m_{k-1}<  s+1< m_{k}$, then it follows that $1\leq s'+1\leq m(k)$. 
 Choose $h_{s'+1}^{k}\in J(s)$ such that $h_{s'+1}^k\in g_{s'+1}^k \Gamma_{k}$. 
 Define $\eta(h_{s'+1}^k\gamma)=1$ and $\eta(g\gamma)=0$ for every $g\in J(s)\setminus\{h_{s'+1}^k\}$ and every $\gamma\in \Gamma_{s+1}$.
 
 If $s+1=m_k$, define $\eta(g\gamma)=0$ for every $g\in J(s)$ and $\gamma\in \Gamma_{s+1}$.

 \medskip

Since $h_{m(k)}^k\in D_{m_k-2}$ for every $k\in\mathbb{N}$, under taking a subsequence of $(\Gamma_n)_{n\in\mathbb{N}}$, we can suppose the following condition:
\begin{align}\label{Linking}
    v^{-1}h_{m(k)}^k\in D_{m_k}\mbox{ for every }v\in (\Gamma_{m_k-2}\cap D_{m_k-1})\setminus\{1_G\}.
\end{align}

\begin{exem}   This example aims to illustrate the previous construction. We consider $[G:\Gamma_1]=[\Gamma_1:\Gamma_2]=[\Gamma_2:\Gamma_3]=3$.  Therefore, ${D_1=\{1_G,g_1^1,g_2^1\}\subseteq G}$, $D_2\cap\Gamma_1=\{1_G,\gamma_1^1,\gamma_2^1\}$, and $D_3\cap \Gamma_2=\{1_G,\gamma_1^2,\gamma_2^2\}$.  Thus, ${D_2=D_1\cup\gamma_1^1 D_1\cup\gamma_2^1 D_1}$, and $D_3=D_2\cup\gamma_1^2 D_2\cup\gamma_2^2 D_2$. In the following figures, the group $G$ is interpreted as an infinite rectangle with infinite cells. 

\begin{figure}[h!]
\begin{tabular}[t]{V{2}c|c|c V{2}}
\hline
     \noalign{\vskip 0 pt}\clineB{1-3}{2}
     \multicolumn{1}{V{3} c |}{\tikzmarknode{D11}{\cellcolor{blue!50}$1$}}
     & & \multicolumn{1}{c V{3}}{} \\\hline
     \noalign{\vskip-1 pt}\clineB{1-3}{3}
     \tikzmarknode{D12}{$1$}&&\\\hline
  \tikzmarknode{D13}{$1$}&&\\\noalign{\vskip 0 pt}\clineB{1-3}{2}
\multicolumn{3}{c}{\smash{\vdots}}\\
\end{tabular}
\caption{In step 1 we have that $\eta(g)=1$ for every $g\in\Gamma_1$. $1_G$ is represented in blue.}
\begin{tikzpicture}[overlay,remember picture]
\node [left=4ex of D11](D11a){$D_1$};
\draw[->] (D11a) -- ([xshift=-2ex]D11.west); 
\node [left=4ex of D12](D12a){$\gamma^1_1D_1$};
\draw[->] (D12a) -- ([xshift=-2ex]D12.west); 
\node [left=4ex of D13](D13a){$\gamma^1_2D_1$};
\draw[->] (D13a) -- ([xshift=-2ex]D13.west);
\draw[decorate,decoration={brace}]
    (D13a.south west) --  ([yshift=2ex]D12a.north west)
    node [midway, left] {\footnotesize{$D_2$}};
\end{tikzpicture}
\end{figure}
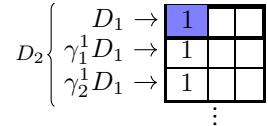

\begin{figure}[h!]
\begin{tabular}[t]{V{2}c|c|c V{2}}
\hline
     \noalign{\vskip 0 pt}\clineB{1-3}{2}
     \multicolumn{1}{V{3} c |}{\tikzmarknode{E11}{$1$}}
     &\cellcolor{blue!50} $0$& \multicolumn{1}{c V{3}}{$\cellcolor{blue!50}0$} \\\hline
     \multicolumn{1}{V{3} c |}{$1$}
    && \multicolumn{1}{c V{3}}{}\\\hline
  \multicolumn{1}{V{3} c |}{\tikzmarknode{E13}{$1$}}
     && \multicolumn{1}{c V{3}}{}\\\hline
     \noalign{\vskip 0 pt}\clineB{1-3}{3}
     \tikzmarknode{E21}{$1$}&$0$&$0$\\\hline
    $1$&&\\\hline
    \tikzmarknode{E23}{$1$}&&\\ \noalign{\vskip 0 pt}\clineB{1-3}{2}
    \tikzmarknode{E31}{$1$}&$0$&$0$\\\hline
    $1$&&\\\hline
    \tikzmarknode{E33}{$1$}&&\\ \noalign{\vskip 0 pt}\clineB{1-3}{2}
\multicolumn{3}{c}{\smash{\vdots}}\\
\end{tabular}
\caption{In step 2 we have that $\eta(d\gamma)=0$ for every $d\in J(1)$ and $\gamma\in\Gamma_2$. In this figure, $J(1)=\{g_1^1,g_2^1\}$ is represented in blue.}
\begin{tikzpicture}[overlay,remember picture]
   \draw[decorate,decoration={brace}]
    ([xshift=-2ex,yshift=-0.5ex]E13.south west) --  ([yshift=0.5ex,xshift=-2ex]E11.north west)
    node [midway, left] {\footnotesize{$D_2$}};
    \draw[decorate,decoration={brace}]
    ([xshift=-2ex,yshift=-0.5ex]E23.south west) --  ([yshift=0.5ex,xshift=-2ex]E21.north west)
    node [midway, left] {\footnotesize{$\gamma_1^2D_2$}};
    
    \draw[decorate,decoration={brace}]
    ([xshift=-2ex,yshift=-0.5ex]E33.south west) --  ([yshift=0.5ex,xshift=-2ex]E31.north west)
    node [midway, left] {\footnotesize{$\gamma_2^2D_2$}};
    
    \draw[decorate,decoration={brace}]
    ([xshift=-8ex,yshift=-0.5ex]E33.south west) --  ([yshift=0.5ex,xshift=-8ex]E11.north west)
    node [midway, left] {\footnotesize{$D_3$}};
\end{tikzpicture}
\end{figure}
\newpage

\begin{figure}[h!]
\begin{tabular}[t]{V{2}c|c|c V{2}}
\hline
     \noalign{\vskip 0 pt}\clineB{1-3}{2}
     $1$
     & $\cellcolor{blue!50}0$&$0$ \\\hline
     $1$
    &$1\cellcolor{red!50}$&$0$ \\\hline
  $1$
     &$0$&$0$ \\\hline
     \noalign{\vskip 0 pt}\clineB{1-3}{2}
     $1$&$0$&$0$\\\hline
    $1$&&\\\hline
   $1$&&\\ \noalign{\vskip 0 pt}\clineB{1-3}{2}
    $1$&$0$&$0$\\\hline
    $1$&&\\\hline
   $1$&&\\ \noalign{\vskip 0 pt}\clineB{1-3}{2}
\multicolumn{3}{c}{\smash{\vdots}}\\
\end{tabular}
\caption{We choose $h_1^1\in D_2$ (red cell) such that $g^1_1\Gamma_1=h_1^1\Gamma_1$. We define $\eta(h_1^1\gamma)=1$ and $\eta(g\gamma)=0$ for every $g\in D_2$ and $\gamma\in\Gamma_3$. $g_1^1$ is represented in blue.}
\end{figure}

\begin{figure}[h!]
\begin{tabular}[t]{V{2}c|c|c V{2}}
\hline
     \noalign{\vskip 0 pt}\clineB{1-3}{2}
     $1$
     & $0$&$\cellcolor{blue!50}0$ \\\hline
     $1$
    &$1$&$0$ \\\hline
  $1$
     &$0$&$0$ \\\hline
     \noalign{\vskip 0 pt}\clineB{1-3}{2}
     $1$&$0$&$0$\\\hline
    $1$&$0$&$0$\\\hline
   $1$&$0$&$\cellcolor{red!50}1$\\ \noalign{\vskip 0 pt}\clineB{1-3}{2}
    $1$&$0$&$0$\\\hline
    $1$&$0$&$0$\\\hline
   $1$&$0$&$0$\\ \noalign{\vskip 0 pt}\clineB{1-3}{2}
\multicolumn{3}{c}{\smash{\vdots}}\\
\end{tabular}
\caption{
We choose $h_2^1\in D_3$ (red cell) such that $g_2^1\Gamma_1=h_2^1\Gamma_1$. We define $\eta(h_2^1\gamma)=1$ and $\eta(g\gamma)=0$ for every $g\in D_3$ and $\gamma\in \Gamma_4$. $g_2^1$ is represented in blue.}
\end{figure}
In Step 5, it is defined $\eta(g\Gamma_5)=0$ for every $g\in D_4$ and $\gamma\in \Gamma_5$. We repeat the process presented in the figures but using  $J(2)$.
\end{exem}

 This construction gives a Toeplitz sequence in $\{0,1\}^G$. Indeed,
let $g\in G$ and $n\in\mathbb{N}$  such that $g\in D_n$. 
If $g\in J(n)$, then our construction applies directly. Otherwise, if $g\notin J(n)$, by definition of $J(n)$, we get that $g\in \bigcup_{i=0}^{n-1} J(i)\Gamma_{i+1}$.

\begin{proposition}\label{period-structure}
    $(\Gamma_n)_{n\in\mathbb{N}}$ is a period structure for the Toeplitz array $\eta$.
\end{proposition}

\begin{proof}
See proof of \cite[Proposition 4.3]{CeCoGo23}.
\end{proof}

\begin{proposition}\label{J-sub}
    Let $\eta\in\{0,1\}^G$ be the Toeplitz array defined above. For every $n$, it holds $J(n)\subseteq\Per(\eta,\Gamma_{n+1})\setminus\Per(\eta,\Gamma_n)$.
\end{proposition}
\begin{proof}
    The construction of $\eta$ implies $J(n)\subseteq \Per(\eta,\Gamma_{n+1})$. 
    
    If $g\in J(n)\cap\Per(\eta,\Gamma_n,0)$, we have that there exists ${h\in J(m_n+l)}$, for some $1\leq l\leq m(n)$, such that $h\in g\Gamma_n$ and $\eta(h)=1$, a contradiction.
    
    If $g\in J(n)\cap\Per(\eta,\Gamma_n,1)$, then $n\neq m_{k-1}-1$ for every $k\in\mathbb{N}$. If $n=m_{k-1}-2$ for some $k\in\mathbb{N}$, then $\gamma g\in J(n+1)$, and hence $\eta(\gamma g)=0$ for every $\gamma\in(\Gamma_n\cap D_{n+1})\setminus\{1_G\}$, a contradiction.
    Therefore, $n= m_{k-1}-1+s'$ for some $k\in\mathbb{N}$ and $1\leq s'< m(k)$. 
    Hence, we can deduce that $g=h_{s'}^k$.
    By the construction of $\eta$ we have that $h_{s'+1}^k\Gamma_k\neq h_{s'}^k\Gamma_k$.
    Therefore, $\gamma g\in J(n+1)$ and $\eta(\gamma g)=0$, where $\gamma\in(\Gamma_n\cap D_{n+1})\setminus\{1_G\}$, and again, we obtain a contradiction. 
    We conclude that $J(n)\cap\Per(\eta,\Gamma_n)=\emptyset$, which implies the Proposition.
\end{proof}

\begin{proposition}\label{T1T2}
    Let $\eta\in\{0,1\}^G$ be the Toeplitz array   previously defined. Let $\mathcal{M}=(n_k)_{k\in\mathbb{N}}$ be the increasing sequence of $\mathbb{N}$ defined by $n_k=m_k-1$, where $m_k$ is defined in (\ref{n-k}) for every $k\in\mathbb{N}$.
    For each $n_k,n_j\in\mathcal{M}$ and $\gamma_0\in \Gamma_{n_k}\cap D_{n_j}$ that satisfies (\ref{good-relation}), we obtain 
\begin{align*}
    \eta(\gamma_0\gamma u)=\eta(u) \mbox{ for every }\gamma\in\Gamma_{n_k}\cap D_{n_k+1}\mbox{ and }u\in J(n_k).
\end{align*}
\end{proposition}
\begin{proof}
  For $g\in J(n_k)$, we have that $\eta(g)=0$.  For each $n_{k},n_{j}\in\mathcal{M}$, with $n_j>n_k$, it is satisfied $n_{j}>n_{k}+2$.
  By Proposition \ref{auxiliar0}, we can conclude that  for all $\gamma_0\in \Gamma_{n_{k}+1}\cap D_{n_{j}}$ satisfying (\ref{good-relation1}), we have that $\gamma_0\gamma u\in J(n_{j})$ for all $\gamma\in (\Gamma_{n_{k}}\cap D_{n_{k}+1})\setminus\{1_G\}$ and $u\in J({n_{k}})$.
  Therefore, we have
\begin{align*}
    \eta(\gamma_0\gamma u)=0=\eta(u).
\end{align*}
If $u\in J(n_k)\subseteq \Per(\eta,\Gamma_{n_k+1})$, then $\eta(\gamma_0 u)=\eta(u)$. Therefore, 
$\eta(\gamma_0\gamma u)=\eta(u)$ for every $\gamma\in\Gamma_{n_k}\cap D_{n_k+1}$ and $u\in J(n_k)$.
\end{proof}
\begin{corollary}\label{no-empty}
    The Toeplitz array $\eta\in\{0,1\}^G$ constructed above is  irregular with $d<1-d$ and $M_G(\overline{O_\sigma(\eta)})\neq \emptyset$.
\end{corollary}
\begin{proof}
    The irregularity of $\eta$ follows from Remark \ref{irregularity} and condition (\ref{LL}). Moreover, Proposition \ref{regular-eta} implies
    \begin{align*}
        d=1-\left(1-\dfrac{1}{|D_1|}\right)\prod_{j=1}^\infty \left(1-\dfrac{|D_j|}{|D_{j+1}|}\right)\leq 1-e^{-2L}<\dfrac{1}{4},
    \end{align*}
    and consequently, $d<1-d$.

The second part of this Corollary follows directly from Propositions \ref{J-sub}, \ref{T1T2}, and \ref{Supp-mu}.
\end{proof}

 Let $\mu\in M_G(\overline{O_\sigma(\eta)})$ be a limit point of the sequence $(\mu_{n_k})_{k\in\mathbb{N}}$, where $\mathcal{M}=(n_k)_{k\in\mathbb{N
 }}$ is defined in Proposition \ref{T1T2}. For each $i\in\{0,1\}$, $[i]=\{x\in\overline{O_\sigma(\eta)}\mid x(1_G)=i\}$ is a clopen set in $\overline{O_\sigma(\eta)}$. Therefore,
\begin{align*}
    \mu([0])&=\lim_{k\to\infty}\dfrac{|J(n_k)|+|\Per(\eta,\Gamma_{n_k},0)\cap D_{n_k}|}{|D_{n_k}|}=1-d+\lim_{k\to\infty}\dfrac{|\Per(\eta,\Gamma_{n_k},0)\cap D_{n_k}|}{|D_{n_k}|},\\
    \mu([1])&=\lim_{k\to\infty}\dfrac{|\Per(\eta,\Gamma_{n_k},1)\cap D_{n_k}|}{|D_{n_k}|}.
\end{align*}
Henceforth, we fix $\mu\in M_G(\overline{O_\sigma(\eta)})$ and we consider $\mathcal{N}=(n_{k_j})_{j\in\mathbb{N}}$ a subsequence of $\mathcal{M}$ such that $(\mu_{n_{k_j}})_{j\in\mathbb{N}}$ converges to $\mu$ when $j\to\infty$.

 \begin{lemma}[{\cite[Lemma 4.10]{CeCoGo23}}]\label{auxiliar}
For every $i\geq 1$ and $\gamma\in \Gamma_i$, there exists $l\geq i$ such that $\gamma J(i)\subseteq J(l)\Gamma_{l+1}$. 
\end{lemma}
 \begin{proposition}\label{Partitions-C}
   Let $k\in\mathbb{Z}^+$. For each $\gamma\in\Gamma_k$, there exists at most one $g\in J(k)$ such that $\eta(\gamma g)=1$.
 \end{proposition}
 \begin{proof}
Let $k\geq 1$ and $\gamma\in \Gamma_k$. By applying Lemma \ref{auxiliar}, there exists $l\geq k$ such that $\gamma J(k)\subseteq J(l)\Gamma_{l+1}$. Since $\eta$ was defined in $J(l)\Gamma_{l+1}$ at the step $l+1$ and there exists at most one $d\in J(l)$ such that $\eta(d\gamma')=1$ for every $\gamma'\in \Gamma_{l+1}$, we conclude.
 \end{proof}

Let us recall the partition of clopen sets of $\overline{O_\sigma(\eta)}$, given by $\{\sigma^{v^{-1}}C_n\mid v\in D_n\}$, where 
 \begin{align*}
     C_n=\{x\in \overline{O_\sigma(\eta)}: \Per(x,\Gamma_n,\alpha)=\Per(\eta,\Gamma_n,\alpha), \mbox{ for every }\alpha\in\{0,1\}\}.
 \end{align*}
 For each $n\geq 1$  and $g\in J(n)$, we define:
\begin{itemize}
    \item $C_{n,g}=\{x\in C_n: x(g)=1\}$.
    \item $ C_n^0=\{x\in C_n: x(g)=0,\mbox{ for every } g\in J(n)\}$.
    \item $C_n^1=\bigcup_{g\in J(n)}C_{n,g}.$
\end{itemize}

It follows that $C_n=C_n^0\cup C_n^1$. 
Consequently,  for every $n\geq 1$, the collection
\begin{align*}
   \mathcal{P}_n= \{\sigma^{v^{-1}}C_n^i: v\in D_n, i\in\{0,1\}\}
\end{align*}
forms a clopen partition of $\overline{O_\sigma(\eta)}$.
 \begin{proposition}\label{at-least}
     The map $\pi: M_G(\overline{O_\sigma (\eta)})\to [0,1]^2, \pi(\mu’)=(\mu’([0]),\mu’([1]))$, is constant.
 \end{proposition}
 \begin{proof}
     For each $n\in\mathbb{N}$ and $i\in\{0,1\}$, let $a_{n,i} =|D_n\cap Per(\eta, \Gamma_n, i)|$.
 We define $C_0^i=[i]\cap \overline{O_\sigma (\eta)}$. Note that
 \begin{align*}
     C_0^0=&\bigcup_{g\in \Per(\eta,\Gamma_n,0)\cap D_n}\sigma^{g^{-1}}C_n\cup \bigcup_{g\in J(n)} \sigma^{g^{-1}}\big(C_n^0\cup \bigcup_{h\in J(n)\setminus\{g\} } C_{n,h}\big).\\
     C_0^1=&\bigcup_{g\in \Per(\eta,\Gamma_n,1)\cap D_n}\sigma^{g^{-1}}C_n\cup \bigcup_{g\in J(n)} \sigma^{g^{-1}}  C_{n,g}.
 \end{align*}
 Therefore, for every $\mu'\in M_G(\overline{O_\sigma (\eta)})$, we have
 \begin{align*}
     \mu’(C_0^0)=&\dfrac{a_{n,0}}{|D_n|}+|J(n)|\mu’(C_{n}^0)+|J(n)|\mu'(C_{n}^1)-\sum_{g\in J(n)}\mu’(C_{n,g}) \\
     =&\dfrac{a_{n,0}}{|D_n|}+\dfrac{|J(n)|}{|D_n|}-\sum_{g\in J(n)}\mu’(C_{n,g})\\
     =&\dfrac{a_{n,0}}{|D_n|}+\dfrac{|J(n)|}{|D_n|}-\mu’(C_n^1)\\
     \mu’(C_0^1)=&\dfrac{a_{n,1}}{|D_n|}+\mu’(C_n^1).
 \end{align*}
 Since $\lim\limits_{n\to\infty}\mu’(C_n^1)=0$, we can conclude
 \begin{align*}
     \mu’(C_0^0)=&\lim\limits_{n\to\infty}\dfrac{a_{n,0}}{|D_n|}+\dfrac{|J(n)|}{|D_n|}=1-d+\lim_{n\to\infty} \dfrac{a_{n,0}}{|D_n|},\\
     \mu’(C_0^1)=&\lim\limits_{n\to\infty}\dfrac{a_{n,1}}{|D_n|}.
 \end{align*}
This completes the proof.
 \end{proof}
 \begin{corollary}\label{equals-partitions}
 Let $\mu'$ be an invariant probability measure of $\overline{O_\sigma(\eta)}$. For every $n\in\mathbb{N}$ and $i\in\{0,1\}$, we have that $\mu'(C_n^i)=\mu(C_n^i)$. 
 \end{corollary}
 \begin{proof}
     From the proof of Proposition 
     \ref{at-least}, we can also deduce the following expressions
     \begin{align*}
         \mu’(C_0^0)=&a_{n,0}(\mu’(C_n^0)+\mu’(C_n^1))+|J(n)|\mu’(C_n^0)+(|J(n)|-1)\mu’(C_n^1)\\
         =&(a_{n,0}+|J(n)|)\mu’(C_n^0)+(a_{n,0}+|J(n)|-1)\mu’(C_n^1),
     \end{align*}
     and
     \begin{align*}
         \mu’(C_0^1)=&a_{n,1}(\mu’(C_n^0)+\mu’(C_n^1))+\mu’(C_n^1)\\
         =&a_{n,1}\mu’(C_n^0)+(a_{n,1}+1)\mu’(C_n^1).
     \end{align*}
     Let
     \begin{align*}
         A_n=\left(\begin{array}{cc}
             a_{n,0}+|J(n)| &a_{n,0}+|J(n)|-1  \\
              a_{n,1}& a_{n,1}+1 
         \end{array}\right).
     \end{align*}
     Note that $A_n$ is an invertible matrix since $\det(A_n)=|D_n|$. Hence, we can deduce that $\mu'(C_n^i)=\mu(C_n^i)$ for every $i\in\{0,1\}$, since the map $\pi$ defined in Proposition \ref{at-least} is constant.    
 \end{proof}

 \section{Topo-isomorphism of $\overline{O_\sigma(\eta)}$}\label{sec:Topo}
 Recall the sequence of natural numbers $\mathcal{M}=(n_k)_{k\in\mathbb{N}}$ defined in  Proposition \ref{T1T2}, $\mu$ the fixed invariant probability measure of $\overline{O_\sigma(\eta)}$, and the subsequence $\mathcal{N}=(n_{k_j})_{j\in\mathbb{N}}$ of $\mathcal{M}$ such that $(\mu_{n_{k_j}})_{j\in\mathbb{N}}$ converges to $\mu$, when $j\to\infty$.
 For every $n\geq 1$, the set $U_n$ was defined in (\ref{definition}) as $U_n=\{x\in\{0,1\}^G: x(D_{n+1})=\eta_n(D_{n+1})\}$.
 
The following Lemma is similar to Lemma 5.1 in \cite{CeCoGo23}.
\begin{lemma}\label{U'ns}
    Let $n_{k_j}\in\mathcal{N}$, with $j\in\mathbb{N}$. Then,
    \begin{align*}
        \mu(U_{n_{k_j}})\geq \lim_{j\to\infty}\dfrac{1}{D_{n_{k_j}+1}} \prod_{l=1}^{j}\left(1-\dfrac{|D_{n_{k_j}+l}|}{|D_{n_{k_j}+1+l}|}\right),
    \end{align*}
    and
    \begin{align*}
        \mu\left(\bigcup_{v\in D_{n_{k_j}+1}}\sigma^{v^{-1}} U_{n_{k_j}}\right)\geq \lim_{j\to\infty}\prod_{l=1}^{j}\left(1-\dfrac{|D_{n_{k_j}+l}|}{|D_{n_{k_j}+1+l}|}\right).
    \end{align*}
\end{lemma}
\begin{proof}
Denote $n=n_{k_{j'}}$ and $m=n_{k_{j}}$ in $\mathcal{N}$ such that  $m>n$. According to Lemma  \ref{good-relation1},
 \begin{align*}
     N_{m,n}\geq \dfrac{|D_{m}|}{|D_{n+1}|}\prod_{l=1}^{m-n-1}\left(1-\dfrac{|D_{n+l}|}{|D_{n+l+1}|}\right).
 \end{align*}
From Lemma \ref{good-patches} and Proposition \ref{T1T2}, we deduce that for every $\gamma\in \Gamma_{n+1}\cap D_m$ satisfying (\ref{good-relation}),
 \begin{align*}
     \eta_n(D_{n+1})=\eta(\gamma D_{n+1})=\eta_{m}(\gamma D_{n+1}).
 \end{align*} 
 Consequently, it follows  that $\sigma^{\gamma^{-1}}\eta_{m}\in U_n$ and $\sigma^{(\gamma v)^{-1}}\eta_{m}\in\sigma^{v^{-1}} U_n$ for all $v\in D_{n+1}$. Therefore, we have
 \begin{align*}
     \mu_{m}(U_n)\geq N_{m,n}\dfrac{1}{|D_{m}|}\geq \dfrac{1}{|D_{n+1}|}\prod_{l=1}^{m-n-1}\left(1-\dfrac{|D_{n+l}|}{|D_{n+l+1}|}\right),
 \end{align*}
 and
 \begin{align*}
     \mu_{m}\left(\bigcup_{v\in D_{n+1}}\sigma^{v^{-1}}U_n\right)\geq N_{m,n}\dfrac{|D_{n+1}|}{|D_{m}|}\geq \prod_{l=1}^{m-n-1}\left(1-\dfrac{|D_{n+l}|}{|D_{n+l+1}|}\right).
 \end{align*}
We observe that $\lim\limits_{j\to\infty}\prod_{l=1}^{j}(1-\frac{|D_{n+l}|}{|D_{n+l+1}|})=\lim\limits_{j\to\infty}\prod_{l=1}^{n_{k_{j}}-n-1}(1-\frac{|D_{n+l}|}{|D_{n+l+1}|})$. Since $\mu$ is the limit of $(\mu_{n_{k_j}})_{j\in\mathbb{N}}$ and $U_n$ is clopen, we conclude the Lemma.
\end{proof}

   Let $k\in\mathbb{N}$. Suppose that $n_{k-1}+1<s< n_{k}+1$, with $n_{k-1},n_k$ two consecutive elements in $\mathcal{M}$.  We can deduce from the definition of $\eta$ that
    \begin{align}\label{Periodo-1}
        \Per(\eta,\Gamma_s,1)= \Gamma_1\cup\bigcup_{i=1}^{k-1}\bigcup_{t=1}^{m(i)} h_t^{i}\Gamma_{n_{i-1}+1+t}\cup\bigcup_{t=1}^{s-(n_{k-1}+1)} h_t^{k}\Gamma_{n_{k-1}+1+t}.
    \end{align}
    We assume that
    $m(0)=1$, and $h_1^0=1_G$. Moreover, since $n_{k-1}+1+m(k)=n_k$ and $\eta(g\gamma)=0$ for every $g\in D_{n_k} $ and $\gamma\in \Gamma_{n_k+1}$, we conclude $$\Per(\eta,\Gamma_{n_{k-1}+1+m(k)},1)=\Per(\eta,\Gamma_{n_k},1)=\Per(\eta,\Gamma_{n_{k}+1},1).$$

\begin{lemma}\label{Good-D's}
    Let $n_{k}\in\mathcal{M}$ be such that $n_{k}\geq 2$. For every $w\in D_{n_{k}-1}\setminus\{1_G\}$, there exists $g_w\in \Per(\eta,\Gamma_{n_{k}-1},1)$ such that  $wg_w\notin \Per(\eta,\Gamma_{n_{k}+1},1)$ and $wg_w\in D_{n_k+1}$. 
\end{lemma}
\begin{proof} First, we prove that for every $w\in D_{n_k-1}\setminus \{1_G\}$ there exists $g_w'\in \Per(\eta,\Gamma_{n_k-1},1)$ such that $wg_w'\notin \Per(\eta,\Gamma_{n_k+1},1)$.
   We proceed by contradiction. Let $w\in D_{n_k-1}\setminus \{1_G\} $ and suppose that for every $g\in \Per(\eta,\Gamma_{n_{k}-1},1)$ it is satisfied that $wg\in \Per(\eta,\Gamma_{n_{k}+1},1)$. 
   We claim that $w\in \Gamma_{n_k-1}$. We know that $\Gamma_1\subset \Per(\eta,\Gamma_1,1)\subset \Per(\eta, \Gamma_{n_{k}-1},1)$. 
   Thus,  $w\Gamma_1\subseteq \Per(\eta,\Gamma_{n_{k}+1},1)$, and consequently, using equation (\ref{Periodo-1}), there exist  $0\leq i\leq k$ and $1\leq l\leq m(i)$ such that ${w\Gamma_1=h_t^i\Gamma_1}$ with $\eta(h_t^i\gamma)=1$ for every $\gamma\in\Gamma_1$. 
   However, according to  Proposition \ref{J-sub},  $h_t^i=h_1^0=1_G$. 
Thus, we have $w\in\Gamma_1$. 

   \medskip
   
 Suppose $w\in \Gamma_{n_{s-1}+(r+1)}$, for some $1\leq s<k,1\leq r\leq m(s)$ or ($s=k$ and $1\leq r< m(k)$) such that $n_{s-1}+(r+1)\leq  n_{k}-2$. 
   
        {\bf Case $1$: }$r<m(s)$.\\
        Since $h_{r+1}^s\in \Per(\eta,\Gamma_{n_{s-1}+(r+1)+1},1)\subset \Per(\eta,\Gamma_{n_k-1},1)$, we deduce that $w h_{r+1}^s\gamma'$ belongs to $\Per(\eta,\Gamma_{n_k+1},1)$ for every $\gamma'\in \Gamma_{n_{s-1}+r+1+1}$.
        Equation (\ref{Periodo-1}) implies $w h_{r+1}^s=h_t^i\gamma'$ for some $0\leq i \leq k$, $1\leq t\leq m(i)$ and $\gamma'\in \Gamma_{n_{i-1}+1+t}$. \\
        If $s<i\leq k$ or  ($i=s$ and $r+1<t\leq m(i)$), we have $n_{s-1}+(r+1)+1\leq n_{i-1}+t$. 
        In this case, we obtain $w h_{r+1}^s\Gamma_{n_{s-1}+1+(r+1)}=h_t^i\Gamma_{n_{s-1}+1+(r+1)}$ with $\eta(h_t^i \gamma)=1$ for every $\gamma\in \Gamma_{n_{s-1}+1+(r+1)}$. 
        In particular, $\eta(h_{t}^i\gamma)=1$ for  $\gamma\in \Gamma_{n_{i-1}+t}$. This is a contradiction since $h_t^i\in J(n_{i-1}+t)$ and $\eta |_{h_t^i\Gamma_{n_{i-1}+t}}$ is not constant by Proposition \ref{J-sub}.
        Therefore, $i<s$ or ($i=s$ and $1\leq t\leq r+1$). \\
        Now, suppose that $i<s$ or ($i=s$ and $1\leq t<r+1$). Thus, $n_{s-1}+(r+1)\geq n_{i-1}+t+1$. $w\in\Gamma_{n_{s-1}+(r+1)}$ implies $h_{r+1}^s\Gamma_{n_{i-1}+1+t}=h_t^i\Gamma_{n_{i-1}+1+t}$. Consequently, $\eta(h_{r+1}^s\gamma'')=1$ for every $\gamma''\in \Gamma_{n_{i-1}+t+1}$ and again, this is impossible by Proposition \ref{J-sub}.
        Thus, $i=s$ and $t=r+1$. We conclude that   $w\in \Gamma_{n_{s-1}+(r+1)+1}$.
    
     {\bf Case $2$:}  $r=m(s)$. Under this assumption, we have that $0\leq s\leq k-1$.\\
        Since $h_1^{s+1}\in\Per(\eta,\Gamma_{n_s+1+1},1)\subset\Per(\eta,\Gamma_{n_k-1},1)$, we have that $wh_1^{s+1}\gamma$ belongs to $\Per(\eta,\Gamma_{n_k+1},1)$ for every $\gamma\in \Gamma_{n_s+1+1}$.
        Thus, $w h_1^{s+1}=h_t^i\gamma'$ for some $0\leq i\leq k-1$, $1\leq t\leq m(i)$, and $\gamma'\in \Gamma_{n_{i-1}+1+t}$.\\ 
       If $i>s+1$ or ($i=s+1$ and $1<t\leq m(s+1)$), we have that $n_{s}+2\leq n_{i-1}+t$.
        Thus, $h_{1}^{s+1}\Gamma_{n_{s}+2}=h_t^i\Gamma_{n_{s}+2}$ with $\eta(h_t^i \gamma)=1$ for every $\gamma\in \Gamma_{n_{s}+2}$, which is impossible by Proposition \ref{J-sub}. 
        Hence, $i< s+1$ or ($i=s+1$ and $t=1$). \\
        Consider the case $i<s+1$. Our hypothesis says that $w\in \Gamma_{n_{s-1}+1+m(s)}$. Since $n_s+1>n_{s-1}+m(s)\geq n_{i-1}+t$, we obtain that 
    $h_{1}^{s+1}\Gamma_{n_{i-1}+1+t}=h_t^i\Gamma_{n_{i-1}+1+t}$, which implies $\eta(h_{1}^{s+1}\gamma'')=1$ for every $\gamma''\in \Gamma_{n_{i-1}+1+t}$.

    Proposition \ref{J-sub} implies that this is only true when $i=s+1$ and $t=1$. Consequently, $w\in \Gamma_{n_{s}+2}$.
    
Applying repeatedly Case 1 and 2, we obtain that $w\in \Gamma_{n_k-1}$.
Therefore, $w\in D_{n_k-1}\cap \Gamma_{n_k-1}=\{1_G\}$, a contradiction.
Hence, there exists $g_w'\in \Per(\eta,\Gamma_{n_k-1},1)$ such that $wg_w'\notin \Per(\eta,\Gamma_{n_k+1},1)$. 

There exists $d\in D_{n_k+1}$ and $\gamma\in \Gamma_{n_k+1}$ such that $wg_w'=d\gamma$. Note that $d\notin\Per(\eta,\Gamma_{n_k+1},1)$. Indeed, if $d\in\Per(\eta,\Gamma_{n_k+1},1)$, it would imply $wg_w'\in \Per(\eta,\Gamma_{n_k+1},1)$. Moreover, $g_w'\gamma^{-1}\in\Per(\eta,\Gamma_{n_k-1},1)$. Therefore, $g_w=g_w'\gamma^{-1}$ satisfies the Lemma.
\end{proof}
For  $n_{k}\in\mathcal{M}$, $k\in\mathbb{N}$, we denote 
\begin{align*}
    Y_{n_{k}}=\bigcap_{\gamma\in\Gamma_{n_{k}}\cap D_{n_{k}+1}}\sigma^\gamma C_{n_{k}}^0.
\end{align*}

\begin{lemma}\label{Containing U_n}
For every $n_{k}\in\mathcal{M}$, $k\in\mathbb{N}$, we have that 
$U_{n_{k}}\cap \overline{O_\sigma(\eta)}\subseteq Y_{n_{k}}.$
\end{lemma}
\begin{proof}
Let $n_{k}\in\mathcal{M}$ for some $k\geq 1$ and $x\in U_{n_{k}}\cap \overline{O_\sigma(\eta)}$. 
Recall that $\{\sigma^{v^{-1}}C_{n_k}\mid v\in D_{n_k}\}$ is a partition of $\overline{O_\sigma(\eta)}$. Therefore, there exist  $y\in C_{n_k}$ and $v\in D_{n_k}$ such that $x=\sigma^{v^{-1}}y$. 
Suppose that $v\neq 1_G$.
There exists $w\in D_{n_k-1}$ satisfying $vw\in \Gamma_{n_k-1}$. If $w=1_G$, then $v\in \Gamma_{n_k-1}\cap D_{n_k}$.  Let $h:=h_{m(k)}^{k}\in J(n_k-1)$ and set $g:=v^{-1}h$, which is in $D_{n_k+1}$ by the condition given in \ref{Linking}. Since $h\in\Per(\eta,\Gamma_{n_k},1)$, we have
\begin{align}\label{vgh}
    y(vg)=y(h)=1.
\end{align}
On the other hand, $g=\gamma d$ for some $\gamma\in\Gamma_{n_k}\cap D_{n_k+1}$ and $d\in D_{n_k}$.

If $d\in \Per(\eta,\Gamma_{n_k})$,  Lemma \ref{Per-eq} implies that there exist $d'\in D_{n_k-1}$ and $\gamma'\in \Gamma_{n_k}$ such that $d=\gamma' d'$. 
Thus, $g=v^{-1}h=\gamma\gamma' d'$. 
Since $h, d'\in D_{n_k-1}$, and $v^{-1},\gamma\gamma'$ are in $\Gamma_{n_k-1}$, we obtain $h=d'$.
For that reason, $v^{-1}=\gamma\gamma'\in\Gamma_{n_k}$, and consequently, $v\in \Gamma_{n_k}\cap D_{n_k}=\{1_G\}$, a contradiction.
Therefore, $d\in J(n_k)$, and this implies $\eta(d)=0$.
Using equation (\ref{vgh}), the previous argument and the fact that $x\in U_{n_k}$, we obtain that
\begin{align*}
1=y(h)=y(vg)=x(g)=x(\gamma d)=\eta(d)=0,
\end{align*}
a contradiction.

Suppose that $w\neq1_G$. By Lemma \ref{Good-D's} there exists $g_w\in \Per(\eta,\Gamma_{n_k-1},1)$ such that $wg_w\notin \Per(\eta,\Gamma_{n_k+1},1)$ and $wg_w\in D_{n_k+1}$. 
Thus, $wg_w\in J(n_k+1)\cup \Per(\eta, \Gamma_{n_k+1},0)$. 
There exist $d\in D_{n_k}$ and $\gamma\in \Gamma_{n_k}\cap D_{n_k+1}$ such that $wg_w=\gamma d$. 
If $wg_w\in J(n_k+1)$, then $d\in J(n_k)$ by Proposition \ref{auxiliar0}. Since $x\in U_{n_k}$, we obtain
\begin{align*}
    x(wg_w)=x(\gamma d)=\eta(d)=0.
\end{align*}
If $wg_w\in\Per(\eta,\Gamma_{n_k+1},0)$, then $d\in\Per(\eta,\Gamma_{n_k},0)\cup J(n_k)$. If $d\in J(n_k)$, we obtain that $\gamma=1_G$ by Proposition \ref{auxiliar0}. We conclude $wg_w\in J(n_k)\subset D_{n_k}$. Since $x\in U_{n_k}$,
\begin{align*}
    x(wg_w)=\eta(wg_w)=0.
\end{align*}
When $d\in \Per(\eta,\Gamma_{n_k},0)$, we have
\begin{align*}
    x(wg_w)=x(\gamma d)=\eta(d)=0.
\end{align*}
In any case, we conclude that
\begin{align}\label{wg0}
    x(wg_w)=0.
\end{align}

 On the other hand, since $g_w\in \Per(\eta,\Gamma_{n_k-1},1)$, $y\in C_{n_k}$ and $vw\in \Gamma_{n_k-1}$, we obtain 
\begin{align}\label{wg1}
    x(wg_w)=y(vw g_w)=1.
\end{align}

We have a contradiction with equations (\ref{wg0}) and (\ref{wg1}). 

Therefore, we conclude that $v=1_G$. Hence, $x\in C_{n_k}$. Note that $x\in U_{n_k}$ implies 
\begin{align*}
    x(\gamma g)=\eta(g)=0 \mbox{ for every }g\in J(n_k)\mbox{ and }\gamma\in \Gamma_{n_k}\cap D_{n_k+1}.
\end{align*}
Thus, $x\in \bigcap_{\gamma\in \Gamma_{n_k}\cap D_{n_k+1}}\sigma^{\gamma} C_{n_k}^0=Y_{n_k}$.
\end{proof}

\medskip

\begin{lemma}\label{Y'n containing}  Let $n_{k_j}\in \mathcal{N}$ with $j\in\mathbb{N}$. Then the following relationship holds
\begin{align*}
    \bigcup_{v\in D_{n_{k_j}+1}}\sigma^{v^{-1}} Y_{n_{k_j}}\subseteq \bigcup_{u\in D_{n_{k_j}}}\sigma^{u^{-1}}C_{n_{k_j}}^0.
\end{align*}    
\end{lemma}
\begin{proof}
    Let $x\in \sigma^{v^{-1}}Y_{n_{k_j}}$ for some $v\in D_{n_{k_j}+1}$. 
    Therefore, there exist $u\in D_{n_{k_j}}$ and $\gamma\in \Gamma_{n_{k_j}}\cap D_{n_{k_j}+1}$ such that $v=\gamma u$.
    The definition of $Y_{n_{k_j}}$ implies that $x\in \sigma^{u^{-1}} C_{n_{k_j}}^0$.
\end{proof}

\medskip

Corollary \ref{no-empty} implies that $\eta$ is irregular. Therefore, Remark \ref{irregularity} implies
\begin{align*}
    \lim_{n\to\infty}\lim_{k\to\infty}\prod_{l=1}^{k-1}\left(1-\dfrac{|D_{n+l}|}{|D_{n+l+1}|}\right)=1.
\end{align*}
For every $n\in\mathbb{N}$, we define
\begin{align*}
    Z_n=\bigcup_{v\in D_n}\sigma^{v^{-1}} C_{n}^0.
\end{align*}

\begin{lemma}\label{Measure-1} Let $\mu'\in M_G(\overline{O_\sigma(\eta)})$. We have that
        $\lim_{j\to\infty} \mu'(Z_{n_{k_j}})=1$.
\end{lemma}
\begin{proof}
    By lemmas \ref{Y'n containing}, \ref{Containing U_n} and \ref{U'ns}, we get
    \begin{align*}
        \mu\left(\bigcup_{u\in D_{n_{k_j}}}\sigma^{u^{-1}}C_{n_{k_j}}^0\right)\geq& \;\mu\left(\bigcup_{v\in D_{n_{k_j}+1}}\sigma^{v^{-1}}Y_{n_{k_j}}\right)\geq \mu\left(\bigcup_{v\in D_{n_{k_j}+1}}\sigma^{v^{-1}}U_{n_{k_j}}\right)\\
        \geq & \lim_{s\to\infty}\prod_{l=1}^{s}\left(1-\dfrac{|D_{n_{k_j}+l}|}{|D_{n_{k_j}+1+l}|}\right).
    \end{align*}
    Therefore, $ \lim_{j\to\infty}\mu(Z_{n_{k_j}})=1$.
 Corollary \ref{equals-partitions} implies that 
 $\mu(C_{n}^i)=\mu'(C_n^i)$ for every 
$C_n^i\in\mathcal{P}_n$, $n\in \mathbb{N}$, $i\in\{0,1\}$.
 In particular, $\mu(Z_{n_{k_j}})=\mu'(Z_{n_{k_j}})$ for every $j\in\mathbb{N}$, and thus, the Lemma is proven.
\end{proof}

\begin{lemma}\label{Containings}
Let $n\geq 1$, $\gamma,\widetilde{\gamma}\in (\Gamma_{n}\cap D_{n+1})\setminus\{1_G\}$ and $g\in J(n)$. We have that
\begin{enumerate}
    \item $\sigma^{\gamma^{-1}} C_{n+1}^0\subseteq C_{n}^0$.
    \item $\sigma^{\gamma^{-1}} C_{n+1,\gamma g}\subseteq C_{n,g}$.
    \item $\sigma^{{\gamma}^{-1}} C_{n+1,\widetilde{\gamma} g}\subseteq C_{n}^0$, when $\gamma\neq \widetilde{\gamma}$.
    \item If $n\in\mathcal{M}$, then $C_{n+1}\subseteq C_{n}^0$.
    \item If $n=n_{k-1}+s'$, with $n_{k-1}\in\mathcal{M}$ and $1\leq s'\leq m(k)$, then $C_{n+1}\subseteq C_{n,h_{s'}^k}$.
\end{enumerate}
\end{lemma}
\begin{proof} 
   Using that $C_{n+1}^0\subseteq C_{n+1}\subseteq C_{n}$, we obtain that $\sigma^{\overline{\gamma}^{-1}} C_{n+1}^0\subseteq C_{n}$  for every $\overline{\gamma}\in\Gamma_{n}$.
   Let $x\in C_{n+1}^0$. 
   Then, for every $g\in J(n)$ we have that $\gamma g\in J(n+1)$. Thus, $\sigma^{\gamma^{-1}}x(g)=x(\gamma g)=0$. Therefore, $\sigma^{\gamma^{-1}} C_{n+1}^0\subseteq C_{n}^0$.

   \medskip
   
   For \textit{(2)} and \textit{(3)}, note that $\sigma^{\gamma^{-1}}C_{n+1}\subseteq C_{n}$.
Let $x\in \sigma^{\gamma^{-1}} C_{n+1,\widetilde{\gamma} g}$.
This implies there exists $y\in C_{n+1,\widetilde{\gamma} g}$ such that $x=\sigma^{\gamma^{-1}}y$. Let $g'\in J(n)$.
Thus, $\gamma g'\in J(n+1)$, and hence, $\sigma^{\gamma^{-1}}y(g')=y(\gamma g')$. Therefore, $\sigma^{\gamma^{-1}} C_{n+1,\widetilde{\gamma} g}\subseteq C_{n,g}$ when $\gamma=\widetilde{\gamma}$, or $\sigma^{{\gamma}^{-1}} C_{n+1,\widetilde{\gamma} g}\subseteq C_{n}^0$  when $\gamma\neq\widetilde{\gamma}$.

\medskip

Suppose that $n\in\mathcal{N}$. 
Proposition \ref{J-sub} implies that $D_{n}\subseteq \Per(\eta,\Gamma_{n+1})$. Thus, if $x\in C_{n+1}$, then $x(g)=\eta(g)=0$ for every $g\in J(n)$. Therefore, $C_{n+1}\subseteq C_{n}^0$, and we obtain \textit{(4)}.

   \medskip

    For the proof of \textit{(5)}, let $x\in C_{n+1}\subseteq C_{n}$, and note that $x(g)=\eta(g)$ for every $g\in J(n)$. 
    Thus, $x\in C_{n,h_{s'}^k}$.
\end{proof}

We denote \begin{align*}
W_n=\bigcup_{v\in D_{n-1}}\sigma^{v^{-1}}\bigcup_{g\in J(n-1)}\bigcup_{\widetilde{\gamma} \in (\Gamma_{n-1}\cap D_n)\setminus \{1_G\}}\bigcup_{\gamma\in (\Gamma_{n-1}\cap D_n)\setminus \{1_G,\widetilde{\gamma}\}} \sigma^{\gamma^{-1}} C_{n,\widetilde{\gamma} g},
\end{align*}
and $[n,m)=\{s\in\mathbb{Z}\mid n\leq s< m\}$ for every $n,m\in\mathbb{N}$.

\begin{lemma}\label{Z's}
 If $n\in\mathcal{M}$, then
        \begin{align}\label{cont-1}
            Z_{n}=Z_{n+1}\cup W_{n+1}\cup\bigcup_{v\in D_{n}}\sigma^{v^{-1}} C_{n+1}^1.
        \end{align}
         When $n\notin\mathcal{M}$,
        \begin{align}\label{cont-2}
            Z_{n}\subseteq Z_{n+1}\cup W_{n+1}.
        \end{align}

    Furthermore, if $n_{k_j},n_{k_{j+1}}$ are two consecutives elements in $\mathcal{N}$ with $n_{k_j}<n_{k_{j+1}}$, then 
    \begin{align}\label{cont-3}
        Z_{n_{k_j}}\subseteq  Z_{n_{k_{j+1}}}\cup\bigcup_{r=n_{k_j}+1}^{n_{k_{j+1}}} W_r\cup\bigcup_{m\in\mathcal{M}\cap[n_{k_j}, n_{k_{j+1}})}\bigcup_{v\in D_{m}}\sigma^{v^{-1}}C_{m+1}^1.
    \end{align}
\end{lemma}
\begin{proof}
Let $n\in\mathcal{M}$. Lemma \ref{Containings} implies that for every $v\in D_n$,
\begin{align*}
    \sigma^{(\overline{\gamma} v)^{-1}}C_{n+1}^0&\subseteq \sigma^{v^{-1}}C_n^0 \mbox{ for each }\overline{\gamma}\in \Gamma_{n+1}\cap D_n,\\
    \sigma^{v^{-1}}C_{n+1}^1&\subseteq \sigma^{v^{-1}} C_n^0.
    \end{align*}
    
Moreover, for $v\in D_n$, $g\in J(n)$, and $\widetilde{\gamma}\in(\Gamma_n\cap D_{n+1})\setminus\{1_G\}$,
    \begin{align*}
    \sigma^{(\gamma v)^{-1}} C_{n+1,\widetilde{\gamma} g}\subseteq \sigma^{v^{-1}}C_n^0 \mbox{ for each } \gamma\in (\Gamma_n\cap D_n)\setminus\{1_G,\widetilde{\gamma}\}.
    \end{align*}
Since $\mathcal{P}_n$ is partition of $\overline{O_\sigma(\eta)}$ and $\mathcal{P}_{n+1}$ is finer than $\mathcal{P}_n$ for every $n\in\mathbb{N}$, we can deduce equation (\ref{cont-1}).

Let $n\notin\mathcal{M}$. Lemma \ref{Containings} implies that for every $v\in D_n$ and $\widetilde{\gamma}\in (\Gamma_n\cap D_{n+1})\setminus\{1_G\}$,  
\begin{align*}
    \sigma^{(\widetilde{\gamma} v)^{-1}} C_{n+1}^0&\subseteq\sigma^{u^{-1}} C_n^0,\\
    \sigma^{(\gamma v)^{-1}}C_{n+1,\widetilde{\gamma}g}&\subseteq\sigma^{v^{-1}} C_n^0 \mbox{ for each }g\in J(n) \mbox{ and }\gamma\in (\Gamma_n\cap D_{n+1})\setminus\{1_g,\widetilde{\gamma}\}.
\end{align*}
Using that $\mathcal{P}_n$ is a partition of $\overline{O_\sigma(\eta)}$ for $n\in\mathbb{N}$, with $\mathcal{P}_{n+1}$ finer than $\mathcal{P}_n$, we obtain equation (\ref{cont-2}).

    Finally, equation (\ref{cont-3}) follows by repeatedly applying   (\ref{cont-1}) and  equation (\ref{cont-2}).
\end{proof}
\begin{corollary}\label{j,s}
    Let $j,s\in\mathbb{N}$ such that $s>j$. Then,
    \begin{align}
        Z_{n_{k_j}}\subseteq  Z_{n_{k_{s}}}\cup\bigcup_{r=n_{k_j}+1}^{n_{k_{s}}} W_r\cup\bigcup_{m\in\mathcal{M}\cap[n_{k_j}, n_{k_{s}})}\bigcup_{v\in D_{m}}\sigma^{v^{-1}}C_{m+1}^1.
    \end{align}
\end{corollary}
\begin{proof}
    This follows directly by applying (\ref{cont-3}).
\end{proof}

\begin{lemma}\label{V's}
    There exists a subsequence $\mathcal{L}=(t_l)_{l\in\mathbb{N}}\subseteq\mathcal{N}$ such that for every ${\mu'\in M_G(\overline{O_\sigma(\eta)})}$, we have that
$\sum_{l=1}^{\infty} \mu'(V_{t_{l}})<\infty$,    where
    \begin{align*}
        V_{t_{l}}=\left(\bigcup_{r=t_{l-1}+1}^{t_{l}}W_r\cup\bigcup_{m\in\mathcal{M}\cap[t_{l-1},t_{l})}\bigcup_{v\in D_m} \sigma^{v^{-1}} C_{m+1}^1\right)\setminus Z_{t_{l}}, l\geq 2, \mbox{ and }V_{t_1}=\emptyset. 
    \end{align*}
\end{lemma}
\begin{proof}    
Let $\{\varepsilon_l\}_{l\in\mathbb{N}}$ be a strictly decreasing sequence of positive real numbers such that $\sum_{j=1}^\infty \varepsilon_j<\infty$.
Lemma \ref{Measure-1} implies that there exists a subsequence $\mathcal{L}=\{t_l\}_{l\in\mathbb{N}}$ of $ \mathcal{N}$ such that for every $l\in\mathbb{N}$ and $r\geq l$,
\begin{align*}
  \mu'(Z_{t_{r}})\geq 1-\varepsilon_{l}.
\end{align*}
Using that $Z_{t_l}\cap V_{t_l}=\emptyset$, we have
\begin{align*}
    1-\varepsilon_{l}+\mu'(V_{t_{l}})\leq \mu'(Z_{t_{l}})+\mu'(V_{t_{l}})\leq 1,
\end{align*}
which implies $\mu'(V_{t_{l}})\leq \varepsilon_{l}$. Hence, 
$
\sum_{j=1}^\infty \mu'(V_{t_{j}})\leq \sum_{j=1}^\infty \varepsilon_j<\infty$.
\end{proof}

\begin{lemma}\label{Measure-Inter-1}
    For every $t_l\in\mathcal{L}$ we have
    \begin{align*}
        Z_{t_l}\subseteq \left(\bigcap_{j=l}^\infty Z_{t_{j}}\right)\cupdot \bigcup_{r=l+1}^\infty V_{t_{r}}, 
    \end{align*}
    where $V_{t_r}$ is defined as in Lemma \ref{V's}.
\end{lemma}
\begin{proof} We claim that for every $l,p\in\mathbb{N}$, it is true that
\begin{align*}
    Z_{t_l}\subseteq \bigcap_{j=l}^{l+p} Z_{t_j}\cupdot \bigcup_{r=l+1}^{l+p} V_r.
\end{align*}
We will apply induction on $p$. 
Corollary \ref{j,s} and the definition of $V_{t_l}$ implies the claim when $p=1$ and every $l\in\mathbb{N}$.

Suppose that for some $p\in\mathbb{N}$ and every  $l\in\mathbb{N}$ we have that
\begin{align*}
    Z_{t_l}\subseteq \left(\bigcap_{j=l}^{l+p}Z_{t_{j}}\right)\cupdot \bigcup_{r=l+1}^{l+p} V_{t_{r}}.
\end{align*}
We need to prove that
\begin{align*}
    Z_{t_l}\subseteq \left(\bigcap_{j=l}^{l+p+1}Z_{t_{j}}\right)\cupdot \bigcup_{r=l+1}^{l+p+1} V_{t_{r}}.
\end{align*}
Corollary \ref{j,s} implies that
 $Z_{t_{l+p}}\subseteq (Z_{t_{l+p}}\cap Z_{t_{l+p+1}})\cupdot V_{t_{l+p+1}}$. Therefore,
 \begin{align*}
     Z_{t_l}&\subseteq \left(\bigcap_{j=l}^{l+p-1}Z_{t_{j}}\cap Z_{t_{l+p}}\right)\cupdot \bigcup_{r=l+1}^{l+p} V_{t_{r}}\\
     &\subseteq \left(\bigcap_{j=l}^{l+p-1}Z_{t_{j}}\cap ((Z_{t_{l+p}}\cap Z_{t_{l+p+1}})\cup V_{t_{l+p+1}})\right)\cupdot \bigcup_{r=l+1}^{l+p} V_{t_{r}}\\
     &\subseteq \left(\bigcap_{j=l}^{l+p+1}Z_{t_{j}}\right)\cup\left( V_{t_{l+p+1}}\cap \bigcap_{j=l}^{l+p-1}Z_{t_{j}}\right)\cupdot \bigcup_{r=l+1}^{l+p} V_{t_{r}}\\
     &\subseteq \left(\bigcap_{j=l}^{l+p+1}Z_{t_{j}}\right)\cupdot \bigcup_{r=l+1}^{l+p+1} V_{t_{r}}.
 \end{align*}
 
Therefore, for every $l\in\mathbb{N}$ and all $k\in\mathbb{N}$ such that $k>l$, we have that
\begin{align*}
    Z_{t_l}\subseteq \left(\bigcap_{j=l}^{k} Z_{t_{j}}\right)\cup\bigcup_{r=l+1}^{\infty} V_{t_r}.
\end{align*}
Consequently,
\begin{align*}
    Z_{t_l}\subseteq \left(\bigcap_{j=l}^{\infty} Z_{t_{j}}\right)\cupdot\bigcup_{r=l+1}^{\infty} V_{t_r}.
\end{align*}

\end{proof}
\begin{lemma}\label{Measure-Inter-12}
    For every $\mu’\in M_G(\overline{O_\sigma(\eta)})$, we establish
    \begin{align}\label{Inter-1}
        \lim_{l\to\infty}\mu’\left(\bigcap_{j=l}^\infty Z_{t_j}\right)=1,
    \end{align}
    and $\mu’(A)=1$, where
    \begin{align*}
         A=\bigcap_{g\in G}\bigcup_{l\in\mathbb{N}}\bigcap_{j=l}^\infty \sigma^{g} Z_{t_j}.
    \end{align*}
\end{lemma}
\begin{proof}
     Lemma \ref{Measure-Inter-1} implies that for every $t_l\in\mathcal{L}$ 
        $$Z_{t_l}\subseteq \left(\bigcap_{j=l}^\infty Z_{t_j}\right)\cupdot\bigcup_{r=l+1}^\infty V_{t_r}.$$
    Therefore,
    \begin{align*}
        \mu'(Z_{t_l})\leq\mu'\left[ \left(\bigcap_{j=l}^\infty Z_{t_j}\right)\cupdot\bigcup_{r=l+1}^\infty V_{t_r}\right]=\mu' \left(\bigcap_{j=l}^\infty Z_{t_j}\right)+\mu'\left(\bigcup_{r=l+1}^\infty V_{t_r}\right)\leq 1.
    \end{align*}
Thus,
\begin{align*}
    \lim_{l\to\infty}\mu'(Z_{t_l})\leq\lim_{l\to\infty}\mu' \left(\bigcap_{j=l}^\infty Z_{t_j}\right)+\lim_{l\to\infty}\mu'\left(\bigcup_{r=l+1}^\infty V_{t_r}\right)\leq 1.
\end{align*}
Since $\sum_{r=1}^\infty \mu'(V_{t_r})$ converges by Lemma \ref{V's}, we obtain that $$\lim_{l\to\infty}\mu'\left(\bigcup_{r=l+1}^\infty V_{t_r}\right)=0.$$
Therefore, Lemma  
\ref{Measure-1} implies equation (\ref{Inter-1}).
 The last part of this Lemma follows directly from (\ref{Inter-1}), since it implies $\mu'(\bigcup_{l\in\mathbb{N}}\bigcap_{j=l}^\infty Z_{t_j})=1$.
\end{proof}
The following Lemma is inspired by \cite[Lemma 5.7]{CeCoGo23}.
\begin{lemma}\label{1-1}
    Let $A$ be as in Lemma \ref{Measure-Inter-12} and let $\pi: \overline{O_\sigma(\eta)}\to \overleftarrow{G}$ be the factor map from $\overline{O_\sigma(\eta)}$ to its associated $G$-odometer as in Proposition \ref{pimap}. Then, $\pi$ is injective when restricted to $A$.
\end{lemma}
\begin{proof}
Let $x,y\in A$ such that $\pi(x)=\pi(y)$. We aim to prove that $x=y$. Let $g\in G$. 
Since $x,y\in A$,  there exists $l_0\in\mathbb{N}$ such that
$\sigma^{g^{-1}} x,\sigma^{g^{-1}}y\in \bigcap_{j=l_0}^{\infty} Z_{t_j}$.
Let $k\geq l_0$.
 $\pi(x)=\pi(y)$ implies there exists $v_k\in D_{t_k}$ such that  $x,y\in\sigma^{v_k^{-1}} C_{t_k}$. 
Therefore, $\sigma^{g^{-1}}x,\sigma^{g^{-1}}\in \sigma^{{(v_kg)}^{-1}}C_{t_k}$. Let $v\in D_{t_k}$ such that $v_kg\in v\Gamma_{t_k}$. This implies $\sigma^{g^{-1}}x,\sigma^{g^{-1}}\in \sigma^{{v}^{-1}}C_{t_k}^0$. Thus, there exist $w,z\in C_{t_k}^0$ such that $\sigma^{g^{-1}}x=\sigma^{v^{-1}}w$ and $\sigma^{g^{-1}} y=\sigma^{v^{-1}}z$. 
Note that $w|_{D_{t_k}}=z|_{D_{t_k}}$ since $w$ and $z$ are in $C_{t_k}^0$. In particular, $w(v)=z(v)$. Therefore, $x(g)=y(g)$. Since this holds for every $g\in G$, we conclude that $x=y$.
\end{proof}

\begin{proposition}\label{Conjugacy}
If  $\mu’\in M_G(\overline{O_\sigma(\eta)})$,  then 
$\pi: \overline{O_\sigma(\eta)}\to\overleftarrow{G}$ is a measure conjugacy of $\mu'$.
\end{proposition}
\begin{proof}
Let $A$ be as in Lemma \ref{Measure-Inter-12}. By Lemma \ref{1-1}, we know that $\pi|_A: A\to\pi(A)$ is a bijective map satisfying $\pi|_A(\sigma^{g}x)=\phi^{g}\pi|_A(x)$. Moreover, \cite[Theorem 2.8]{Gl03} implies that $\pi|_A$ and $\pi|_A^{-1}$ are measurable. Recall $\overleftarrow{G}$ is uniquely ergodic with unique invariant measure $\nu$. Thus,  $\nu(B)=\mu'(\pi|_A^{-1}(B))$ for every  $B\subseteq \pi(A)$ Borel set. This implies that $\pi|_A$ is a measure conjugacy of $\mu'$.
\end{proof}

The following Proposition is a well-known fact. We provide a proof for completeness.
\begin{proposition}\label{Ergodic-topo}
    Let $(X,\sigma,\mu')$ and $(Y,\phi,\nu)$  be two measure conjugate dynamical systems, where the measure conjugacy is given by $\pi': X'\to Y'$ with $X'\subseteq X$, $Y'\subseteq Y$ 
 invariant Borel sets. Then $\mu'$  is ergodic if and only if $\nu$ is ergodic.
\end{proposition}
\begin{proof}
The measure conjugacy implies that for every Borel set $B\subseteq Y$, we have  $\nu(B)=\mu'(\pi'^{-1}(B))$. Suppose that $\mu'$ is ergodic and
     let $B\subseteq Y$ an invariant Borel set. Observe that $C=\pi'^{-1}(B\cap Y')$ is an invariant Borel set.
    Thus,
    \begin{align*}
        \nu(B)=\nu(B\cap Y')=\mu'(\pi'^{-1}(B\cap Y'))=\mu'(C)\in \{0,1\}.
    \end{align*}
    Therefore, $\nu$ is an ergodic measure. We can interchange the roles of $\pi'$ and $\pi'^{-1}$ to obtain that $\nu$ ergodic implies $\mu'$ ergodic.
\end{proof}

\begin{proof}[Proof of Theorem \ref{theo:main1}] 
 Recall that if we take a subsequence $(\Gamma_{n_j})_{j\in\mathbb{N}}\subseteq (\Gamma_n)_{n\in\mathbb{N}}$, then the $G$-odometers associated coincide, i.e., $\overleftarrow{G}=\varprojlim(G/\Gamma_n,\varphi_n)=\varprojlim(G/\Gamma_{n_j},\varphi_{n_j})$ (See \cite{CoPe08}).  
 Therefore, possibly by taking a subsequence of $(\Gamma_n)_{n\in\mathbb{N}}$, we can suppose that $(\Gamma_n)_{n\in\mathbb{N}}$ satisfies condition (\ref{LL}). 
 Let $\eta\in\{0,1\}^G$ be the Toeplitz sequence defined in Section \ref{sec:irregular-construction}. Proposition \ref{T1T2} ensures that $\eta$ is irregular, and Proposition \ref{period-structure} implies that $\overleftarrow{G}$ is the maximal e\-qui\-con\-ti\-\-nuous factor of $\overline{O_\sigma(\eta)}$. On the other hand, Proposition \ref{Conjugacy} and Proposition \ref{Ergodic-topo} imply that $\overline{O_\sigma(\eta)}$ is uniquely ergodic. Consequently,  $\overline{O_\sigma(\eta)}$ is a topo-isomorphic  extension of its maximal e\-qui\-con\-ti\-\-nuous factor $\overleftarrow{G}$ by Proposition \ref{Conjugacy}.
\end{proof}

\begin{proof}[Proof of Corollary \ref{coro:main2}] This follows directly from Theorem \ref{theo:main1}, since assuming that $G$ is amenable, we can apply Theorem \ref{FuGrLe}. 
    
\end{proof}

\subsection*{Acknowledgements} 
The author gratefully thanks María Isabel Cortez for her invaluable guidance and insightful comments throughout this work. Additionally, the author is grateful with Maik Gröger and Till Hausser for meaningful conversation related to this research.


\begin{thebibliography}{MMM}


\bibitem{Au88} J. Auslander, {\it Minimal flows and their extensions}, North-Holland Mathematics Studies, 153, North-Holland, Amsterdam, 1988.

\bibitem{BeOh07} Y. Benoist and H. Oh, {\it Equidistribution of rational matrices in their conjugacy classes}. Geom. Funct. Anal. {\bf 17} (2007)  no. ~1, 1--32.

\bibitem{CC10} T. Ceccherini-Silberstein and M. Coornaert, {\it Cellular automata and groups}. Springer Monographs in Mathematics. Springer-Verlag, Berlin, 2010.

\bibitem{CeCoGo23} P. Cecchi-Bernales, M. I. Cortez and J. Gómez, {\it Invariant measures of Toeplitz subshifts on non-amenable groups,} arXiv preprint arXiv: 2305.09835 (2023)

\bibitem{Co06} M.I. Cortez, {\it ${\Bbb Z}^d$ Toeplitz arrays}, Discrete Contin. Dyn. Syst. {\bf 15} (2006), no.~3, 859--881.

\bibitem{CoPe08} M.I. Cortez and S. Petite, {\it $G$-odometers and their almost $1$-$1$ extensions} J. Lond. Math. Soc. (2) {\bf 78} (2008), no.~1, 1--20.

\bibitem{CoPe14} M.I Cortez and S. Petite, {\it Invariant measures and orbit equivalence for generalized Toeplitz subshifts};  Groups, Geometry and Dynamics (8), 2014.

\bibitem{Do05} T. Downarowicz, {\it Survey of odometers and Toeplitz flows}. Algebraic and topological dynamics, 7--37, Contemp. Math., 385, Amer. Math. Soc., Providence, RI, 2005.

\bibitem{DoGl16} T. Downarowicz and E. Glasner, {\it Isomorphic extensions and applications}, Topol. Methods Nonlinear Anal. {\bf 48} (2016), no.~1, 321--338.

\bibitem{DoKa15} T. Downarowicz and S. Kasjan, {\it Odometers and Toeplitz systems revisited in the context of Sarnak's conjecture}, Studia Math. {\bf 229} (2015), no.~1, 45--72. 

\bibitem{DoLa98} T. Downarowicz and T. Lacroix, {\it Almost 1-1 extensions of Furstenberg-Weiss type and applications to Toeplitz flows}, Studia Math. (2) {\bf 130} (1998), 149–170.

\bibitem{FuGrLe22} G. Fuhrmann, M. Gr\"{o}ger and D. Lenz, {\it The structure of mean e\-qui\-con\-ti\-\-nuous group actions}, Israel J. Math. {\bf 247} (2022), no.~1, 75--123. 

\bibitem{GJ00} R. Gjerde and {\O}. Johansen, {\it Bratteli-Vreshik models for Cantor minimal systems: applications to Toeplitz flows}, Ergodic Theory and Dynamical Systems \textbf{20} (2000), 1687--1710. 

\bibitem{Gl03} E. Glasner, {\it Ergodic theory via joinings}, Mathematical Surveys and Monographs, 101, Amer. Math. Soc., Providence, RI, 2003.

\bibitem{JaKe69} K. Jacobs and M. Keane, {\it $0$-$1$-sequences of Toeplitz type}, Z. Wahrsch. Verw. Gebiete {\bf 13} (1969), 123--131. 

\bibitem{Ke11} A. Kechris, {\it Classical descriptive set theory}, Graduate Texts in Mathematics, 156, Springer, New York, 1995. 

\bibitem{Kr10} F. Krieger, {\it  Toeplitz subshifts and odometers for residually finite groups}, in {\it \'{E}cole de Th\'{e}orie Ergodique}, 147--161, S\'{e}min. Congr., 20, Soc. Math. France, Paris, 2010. 

\bibitem{LaSt18}  M. \L\c acka and M.  Straszak, {\it Quasi-uniform convergence in dynamical systems generated by an amenable group action}, J. Lond. Math. Soc. (2) {\bf 98} (2018), no.~3, 687--707.

\bibitem{LiTuYe15} J. Li, S. Tu and X. Ye, {\it Mean equicontinuity and mean sensitivity}, Ergodic Theory Dynam. Systems {\bf 35} (2015), no.~8, 2587--2612.

\bibitem{Ve70} W. Veech, {\it Point-distal flows}, Amer. J. Math. {\bf 92} (1970), 205--242. 

\bibitem{Wi84} S. Williams, {\it Toeplitz minimal flows which are not uniquely ergodic}; Zeitschrift f\"ur Wahrscheinlichkeitstheorie und Verwandte Gebiete (67), 1984.


\end{thebibliography}
\end{document}